\documentclass{article}
\usepackage{graphicx} 

\usepackage[usenames,dvipsnames]{xcolor} 
\usepackage{amsmath,amsthm,amssymb,amsfonts} 
\usepackage{mathtools} 
\usepackage{yfonts} 
\usepackage[T1]{fontenc} 
\usepackage[utf8]{inputenc}\DeclareUnicodeCharacter{2212}{-} 
\usepackage{bbm} 
\usepackage[outline]{contour}

\usepackage{float} 

\usepackage{fix-cm} 
\usepackage{lmodern} 
\usepackage[explicit]{titlesec} 

\usepackage[%
  margin=1in, bindingoffset=0.25in]{geometry} 
\usepackage{pdfpages} 
\usepackage{import} 
\usepackage{parskip} 

\usepackage{lipsum} 

\usepackage{todonotes} 

\usepackage{hyperref} 

\usepackage{refcount} 
\usepackage{xstring} 
\usepackage{regexpatch} 

\usepackage{scalerel} 
\usepackage{graphicx} 
\usepackage{pgf} 
\usepackage{eso-pic} 
\usepackage{transparent} 

\usepackage{environ} 
\usepackage[many]{tcolorbox} 
\tcbuselibrary{skins}
\usepackage{enumitem} 
\usepackage{thmtools} 

\usepackage{xifthen} 

\usepackage{booktabs} 
\usepackage{diagbox} 
\usepackage{tabularx} 
\usepackage{multirow} 
\usepackage{colortbl} 
\usepackage{arydshln} 

\usepackage{subcaption} 

\usepackage{tikz-cd}

\usepackage{cleveref}
\usepackage[
    backend=biber,
    style=numeric,
    citestyle=numeric,
    doi=false,
    isbn=false,
    url=false,
    maxnames=50,
    firstinits=true]{biblatex}
\usepackage[ruled, algosection]{algorithm2e}

\SetCommentSty{bluefont}

\declaretheoremstyle[
  spaceabove=\parsep,
  spacebelow=0,
  headfont=\bfseries,
  notefont=\bfseries,
  notebraces={(}{)},
  bodyfont=\normalfont,
  postheadspace=.5em
]{definition}

\newtheoremstyle{plain}         
    {\parsep}                   
    {}                          
    {\itshape}                  
    {}                          
    {\bfseries}                 
    {.}                         
    {.5em}                      
    {\thmname{#1}\thmnumber{ #2}\thmnote{ \bfseries (#3)}}                 

\theoremstyle{plain}
\declaretheorem[sharenumber=algocf]{theorem}
\declaretheorem[sharenumber=algocf]{lemma}

\declaretheorem[sharenumber=algocf]{proposition}

\declaretheorem[sharenumber=algocf]{definition}

\declaretheorem[sharenumber=algocf]{remark}

\declaretheorem[style=algocf, numbered=no,name=Theorem]{thm}

\declaretheorem[style=algocf, numbered=no,name=Approach]{approach}

\newtheoremstyle{proc}{\parsep}                   
{}                          
{\itshape}                  
{}                          
{\bfseries}                 
{.}                         
{.5em}                      
{Procedure\thmnumber{ #2}\thmnote{ \bfseries (#3)}}
\theoremstyle{proc}

\makeatletter
\long\def\@secondofthree#1#2#3{#2}
\long\def\@thirdoffour#1#2#3#4{#3}

\NewEnviron{restatethis}[2]{%
  \begin{#1}%
      \BODY
      \protected@write\@auxout{}{%
        \string\@restatetheorem{#1}{#2}{\csname the#1\endcsname}{\detokenize\expandafter{\BODY}}%
      }%
  \end{#1}%
}

\@ifpackageloaded{thmtools}{%
    \def\restatethm@getthmcountercsname#1{\def\thethmcsname{#1}}%
}{%
    \@ifpackageloaded{ntheorem}{%
        \def\restatethm@getthmcountercsname#1{%
            \def\thethmcsname{\expandafter\expandafter\expandafter\restatethm@ntheorem@getthmcountercsname@helper\csname mkheader@#1\endcsname}}%
        \def\restatethm@ntheorem@getthmcountercsname@helper#1\@thm#2#3#4{#3}
    }{%
        \@ifpackageloaded{amsthm}{%
            \def\restatethm@getthmcountercsname#1{\edef\thethmcsname{\expandafter\expandafter\expandafter\@thirdoffour\csname#1\endcsname}}%
        }{%
            \def\restatethm@getthmcountercsname#1{\edef\thethmcsname{\expandafter\expandafter\expandafter\@secondofthree\csname#1\endcsname}}%
        }%
    }%
}

\newcommand{\@restatetheorem}[4]{%
  \expandafter\gdef\csname restatethis@#2\endcsname{%
    \begingroup
    \restatethm@getthmcountercsname{#1}
    \expandafter\def\csname the\thethmcsname\endcsname{#3}%
    \begin{#1}#4\end{#1}%
    \endgroup
  }%
}
\newcommand{\restate}[1]{\csname restatethis@#1\endcsname} 
\makeatother

\makeatletter
\providecommand*{\cupdot}{%
  \mathbin{%
    \mathpalette\@cupdot{}%
  }%
}
\newcommand*{\@cupdot}[2]{%
  \ooalign{%
    $\m@th#1\cup$\cr
    \sbox0{$#1\cup$}%
    \dimen@=\ht0 %
    \sbox0{$\m@th#1\cdot$}%
    \advance\dimen@ by -\ht0 %
    \dimen@=.5\dimen@
    \hidewidth\raise\dimen@\box0\hidewidth
  }%
}

\providecommand*{\bigcupdot}{%
  \mathop{%
    \vphantom{\bigcup}%
    \mathpalette\@bigcupdot{}%
  }%
}
\newcommand*{\@bigcupdot}[2]{%
  \ooalign{%
    $\m@th#1\bigcup$\cr
    \sbox0{$#1\bigcup$}%
    \dimen@=\ht0 %
    \advance\dimen@ by -\dp0 %
    \sbox0{\scalebox{2}{$\m@th#1\cdot$}}%
    \advance\dimen@ by -\ht0 %
    \dimen@=.5\dimen@
    \hidewidth\raise\dimen@\box0\hidewidth
  }%
}
\makeatother

\makeatletter
\renewenvironment{proof}[1][\proofname]{\par
  \pushQED{\qed}%
  \normalfont \topsep6\p@\@plus6\p@\relax
  \trivlist
  \item[\hskip\labelsep
        \itshape
    #1\@addpunct{.}]\mbox{}\\*
}{%
  \popQED\endtrivlist\@endpefalse
}
\makeatother

\makeatletter
\newcommand{\algorithmstyle}[1]{\renewcommand{\algocf@style}{#1}}
\makeatother

\makeatletter
\newcommand{\nosemic}{\renewcommand{\@endalgocfline}{\relax}}
\newcommand{\dosemic}{\renewcommand{\@endalgocfline}{\algocf@endline}}
\makeatother

\SetKwComment{Comment}{\# }{}

\newcommand\R{\mathbb R}
\newcommand\Z{\mathbb Z}
\newcommand\N{\mathbb N}

\renewcommand{\phi}{\varphi}
\renewcommand{\epsilon}{\varepsilon}

\newcommand{\stab}{\mathrm{Stab}}

\newcommand{\vol}{\mathrm{vol}}

\newcommand{\T}{\mathcal{T}}
\renewcommand{\P}{\mathcal{P}}
\renewcommand{\L}{\mathcal{L}}

\DeclareMathOperator*{\E}{E}
\DeclareMathOperator*{\OO}{O}
\DeclareMathOperator*{\Id}{Id}
\DeclareMathOperator*{\SE}{SE}
\DeclareMathOperator*{\argmin}{argmin} 

\title{An algorithmic approach for computing fundamental domains of crystallographic groups}
\author{
    Reymond Akpanya\thanks{RWTH Aachen University, Chair of Algebra and Representation Theory, Pontdriesch 10-12, 52062 Aachen, Germany. Email: \texttt{reymond.akpanya@rwth-aachen.de}}
\thanks{School of Mathematics and Statistics, The University of Sydney, Carslaw Building F07,
Camperdown NSW 2006, Australia. E-mail: {\tt reymond.akpanya@sydney.edu.au}.}\and 
    Alice C. Niemeyer\thanks{RWTH Aachen University, Chair of Algebra and Representation Theory, Pontdriesch 10-12, 52062 Aachen, Germany. Email: \texttt{alice.niemeyer@momo.math.rwth-aachen.de}} \and Lukas Schnelle \thanks{RWTH Aachen University, Chair of Algebra and Representation Theory, Pontdriesch 10-12, 52062 Aachen, Germany. Email: \texttt{lukas.schnelle1@rwth-aachen.de}, corresponding author} 
}
\date{}

\begin{document}

\maketitle


\begin{abstract}
A crystallographic group is a discrete subgroup of the Euclidean group $\operatorname{E}(n)$ that has a compact fundamental domain. Since such a crystallographic group $\Gamma$ is infinite, computing fundamental domains of $\Gamma$ is algorithmically challenging.
We address this difficulty by targeting the computation of Dirichlet cells that can form fundamental domains of $\Gamma$. We show that the half-spaces defining such a Dirichlet cell can be derived from elements of $\Gamma$ acting on $\mathbb{R}^n$ that can be expressed as words of bounded length in a suitable generating set. Based on these results, we design an algorithm for the computation of fundamental domains of crystallographic groups and exploit it to study the construction of topological interlocking assemblies.
\end{abstract}



\section{Introduction}
One definition of a crystallographic group $\Gamma \leq \E(n)$ is that there exists a compact \emph{fundamental domain}  $F \subset \mathbb{R}^n$ of $\Gamma$, a compact subset which contains a system of representatives $F^\circ \subseteq V \subseteq F$ for the natural action of $\Gamma$ on $\mathbb{R}^n$. Hence, a fundamental domain $F$ of $\Gamma$ allows us to tile the Euclidean space $\mathbb{R}^n$ through the action of $\Gamma$ on $F$.
This property makes crystallographic groups interesting not only from a purely mathematical perspective but also from an artistic one. For instance, the Dutch artist M.\ C.\ Escher exploited crystallographic groups in dimension $2$ to create artworks that continue to fascinate both mathematicians and artists to this day (see \cite{escher_fish_1954} for examples of his work).
In 1912, Bieberbach proved that there is only a finite number of crystallographic groups $\Gamma\leq \E(n)$ for every $n\in \mathbb{N}$, see~\cite{bieberbach1912groups}. 
Furthermore, for $n\leq6$ all crystallographic groups of dimension $n$ are known and enumerated (up to isomorphism). These groups, together with additional structural information such as standard generating sets, can be found in \cite{brown1978list}.
For instance, for $n=2$ there exist exactly $17$ crystallographic groups (wallpaper groups) and for $n=3$ there are exactly $230$ crystallographic groups (space groups).

Since the definition of a crystallographic group requires the existence of a compact fundamental domain, it is natural to ask how such a fundamental domain can be computed for a given crystallographic group $\Gamma\leq \E(n)$. This problem is the focus of this paper.
 If $\Gamma$ arises from a suitable generating set, we are able to answer this question by presenting an algorithm that computes Dirichlet cells forming fundamental domains of $\Gamma$. (The choice of a suitable generating set of $\Gamma$ is described in \Cref{not:cryst-grp-cosets}.)
A Dirichlet cell (also referred to as Voronoi domain) is a convex set in $\mathbb{R}^n$ that can be constructed as the intersection of half spaces $H^+(u,w):=\{v\in \mathbb{R}^n\mid \Vert u-v\Vert \leq \Vert w-v\Vert \}$ between points $u$ and $w$ in $\mathbb{R}^n$, see \Cref{def:dirichlet-cell}. Loosely speaking, $H^+(u,w)$ contains every point $v\in \mathbb{R}^n$ that is closer to $u$ than to $w$ with respect to the Euclidean norm.
For a given crystallographic group $\Gamma\leq \E(n)$ and a point $u\in \mathbb{R}^n$ satisfying $\vert \stab_{\Gamma}(u)\vert =1$, the intersection of all the half spaces $H^+(u,w),$ where $w\in u^\Gamma$, is a Dirichlet cell $D(u,u^\Gamma)$ forming a fundamental domain of $\Gamma$ (see Definition~\ref{def:dirichlet-cell} and Appendix~\ref{apx:thm:dirichlet-is-fund-dom}). Although $\vert u^\Gamma \vert = \infty,$ there exist finitely many $w_1,\ldots,w_\ell\in u^\Gamma$ such that $D(u,u^\Gamma)=D(u,\{w_1,\ldots,w_\ell\}).$
In this work, we present an algorithmic approach for the construction of these elements by solving the word problem: We show that the points $w_1,\ldots,w_\ell$ can be obtained from elements in $\Gamma$ that can be written as words of bounded length in a suitable generating set of $\Gamma.$
In particular, we prove the following theorem.

\begin{thm}[see Theorem \ref{thm:dirichlet-word-length}]
Let $\Gamma\leq \E(n)$ be a crystallographic group generated by a finite set $S$ as described in \Cref{not:cryst-grp-cosets}. Further, let $u\in \mathbb{R}^n$ be a point with $\vert \stab_{\Gamma}(u) \vert =1. $
Then there exists an $A\in \mathbb{N}$ and a finite set $M\subset \Gamma$ such that
(1) every $\phi\in M$ can be written as a word of length at most $A$ in $S$ and (2) $D(u,u^\Gamma)=D(u,u^M).$
\end{thm}
Note, existing algorithms to compute fundamental domains for a given crystallographic group rely on heuristic estimates for the number of elements in $u^\Gamma$ that must be considered to determine $D(u, u^\Gamma)$, see \cite{MR2025179}. Here, we take a different approach by establishing the upper bound on the length of words using a suitably chosen generating set of $\Gamma$ that need to be considered for this computation.
Using this theorem together with some fundamental results, we are able to present Algorithm~\ref{alg:dirichletcell} to compute fundamental domains for crystallographic groups and employ it to compute some fundamental domains of crystallographic groups of dimension $3$. 
Moreover, we use these theoretical results and illustrate an approach for the construction of geometries in $\mathbb{R}^3$ that have the potential to generate topological interlocking assemblies. 
Note, our algorithm relies on the following assumptions:
\begin{enumerate}
    \item The knowledge of a point \( u \in \mathbb{R}^n \) with \( \lvert \stab_{\Gamma}(u) \rvert = 1 \), and
    \item Knowledge of a generating set of \( \Gamma \), as described in \Cref{not:cryst-grp-cosets}.
\end{enumerate}
In our investigations, we apply the algorithm to crystallographic groups $\Gamma \le \E(n)$ with $n=3$. Groups of $n\leq 4$ which satisfy the requirements of \Cref{not:cryst-grp-cosets} can be found in \cite{brown1978list}. Moreover, for any crystallographic group the stabiliser of almost every point $u \in \mathbb{R}^n$ is trivial. Hence, the assumptions required for our method are naturally satisfied in this setting.

This paper is structured as follows: In Section~\ref{section:prelim} we introduce notions on crystallographic groups and Dirichlet domains that are essential for the understanding of this paper.
Next, we derive the theory to establish our algorithm for the computation of fundamental domains in \Cref{section:comp}.
For this, we have to investigate points in general position for a given crystallographic group (see \Cref{section:prelim}) and their resulting Dirichlet domains. We exploit our results to establish \Cref{thm:dirichlet-word-length}. In order to present our proposed algorithm, we further comment on the volume of a fundamental domain of $\Gamma.$
Lastly, we discuss the construction of topological interlocking assemblies, see \Cref{section:TIA}. Inspired by the constructions in \cite{goertzen2024constructinginterlockingassembliescrystallographic,goertzen2022topological} we give a construction of three dimensional geometries that have a high potential of facilitating topological interlocking assemblies. 

Note, it appears that some of the results used in this work are known to the community, although we could not find explicit references. For completeness and clarity, we include proofs of these results in the appendix.

\section{Preliminaries}
\label{section:prelim}
In order to introduce the framework of this work and present our results, we use  
the following definitions:
Let $n\geq 2$ be a natural number and $v,w$ elements in $\mathbb{R}^n$.
Then we denote the Euclidean scalar product of $v$ and $w$ by $\langle v, w \rangle$ and define the \emph{Euclidean norm} of $v$ by $\Vert v \Vert \coloneqq \sqrt{\langle v, v \rangle} $. As usual, the norm $\Vert \cdot \Vert$ induces a metric on $\R^n$ via $d(v, w) \coloneqq \Vert v - w \Vert$.	
For $r\in \R_{>0}$ we define the open \emph{$r$-ball of $v$} as $B_r(v) \coloneqq \{ w \in \R^n \mid d(v, w) < r \}.$
We call a map $\phi : \R^n \to \R^n$ an \emph{isometry}, if it is distance preserving, i.e.\ $d(v^\phi, w^\phi) = d(v, w)$ for all $v, w \in \R^n$.
The set $\E(n)$ of all isometries of $\R^n$ forms a group called the \emph{Euclidean group} of dimension $n.$
A subset $F\subseteq \R^n$ is called a fundamental domain of a group $\Gamma\leq \E(n)$, if (i) $\bigcup_{\gamma \in \Gamma} F^{\gamma} = \R^n$ and (ii) there is a system of representatives $V \subseteq \R^n$ of the orbits of $\Gamma$ on $\R^n$ such that $F^\circ \subseteq V \subseteq F.$
The group $\Gamma$ is called a \emph{crystallographic group}, if $\Gamma$ is a discrete subgroup of $\E(n)$ (this means that $v^\Gamma$ is discrete in $\R^n$ for all $v \in \R^n$) and there exists a compact fundamental domain of $\Gamma$.
Note, that fundamental domains of given crystallographic groups do not have to be connected. However, we assume all fundamental domains in this paper to be connected.
We refer the reader to \cite{eick2006} 
 for further theoretical background and algorithmic methods for the investigation of crystallographic groups. 

Note, the Euclidean group $\E(n)$ is isomorphic to the semidirect product $ \OO(n) \ltimes \R^n$, where $\OO(n)$ denotes the orthogonal group of dimension $n$. 
The existence of an isomorphism $\Psi:\E(n)\to \OO(n) \ltimes \R^n$ allows us to interpret the isometry $\phi\in \E(n)$ as a tuple $\Psi(\phi)=(\phi_o,\phi_t)\in \OO(n) \ltimes \R^n,$ where we call $\phi_o \in \OO(n)$ the \emph{orthogonal component} and $\phi_t \in \R^n$ the \emph{translational component} of $\phi$. The action of $(\phi_o,\phi_t)$ on a point $u \in \R^n$ is given by $u^{(\phi_o, \phi_t)} = u\cdot{\phi_o} + \phi_t$. Here, we choose $\Psi$ to be an isomorphism such that $v^\phi = v^{\Psi(\phi)}$ holds for all $v\in \mathbb{R}^n.$
    Examples of isometries of the Euclidean space $\mathbb{R}^2$ are given by the elements $\phi,\psi\in \E(2)$ defined as follows:
    $$
     \Psi(\phi):=   \left( \begin{pmatrix} -1 &0 \\ 0 &-1\end{pmatrix}, \begin{pmatrix} 0 \\ 0 \end{pmatrix}\right),
     \Psi(\tau_1):=   \left( \begin{pmatrix} 1 &0 \\ 0 &1\end{pmatrix}, \begin{pmatrix} 1 \\ 0 \end{pmatrix}\right),
     \Psi(\tau_2):=   \left( \begin{pmatrix} 1 &0 \\ 0 &1\end{pmatrix}, \begin{pmatrix} 0 \\ 1 \end{pmatrix}\right).
    $$
These two elements generate the wallpaper group $p2=\langle \phi, \tau_1, \tau_2 \rangle$ (using the naming from international notation \cite{ITASubperiodic}).
For a better understanding of the action of this group, we give Figure \ref{fig:pg-action}. This figure illustrates the action of $p2$ on $\R^2.$ If $M$ is the set of black dots forming the letter $\beta$ in the upper rectangle situated in the middle of the described figure, then $M^{p2}$ is given by all the black $\beta$ tiling the plane as indicated in Figure~\ref{fig:pg-action}. 

\begin{figure}[H]
    \centering
    \includegraphics[width=0.35\textwidth]{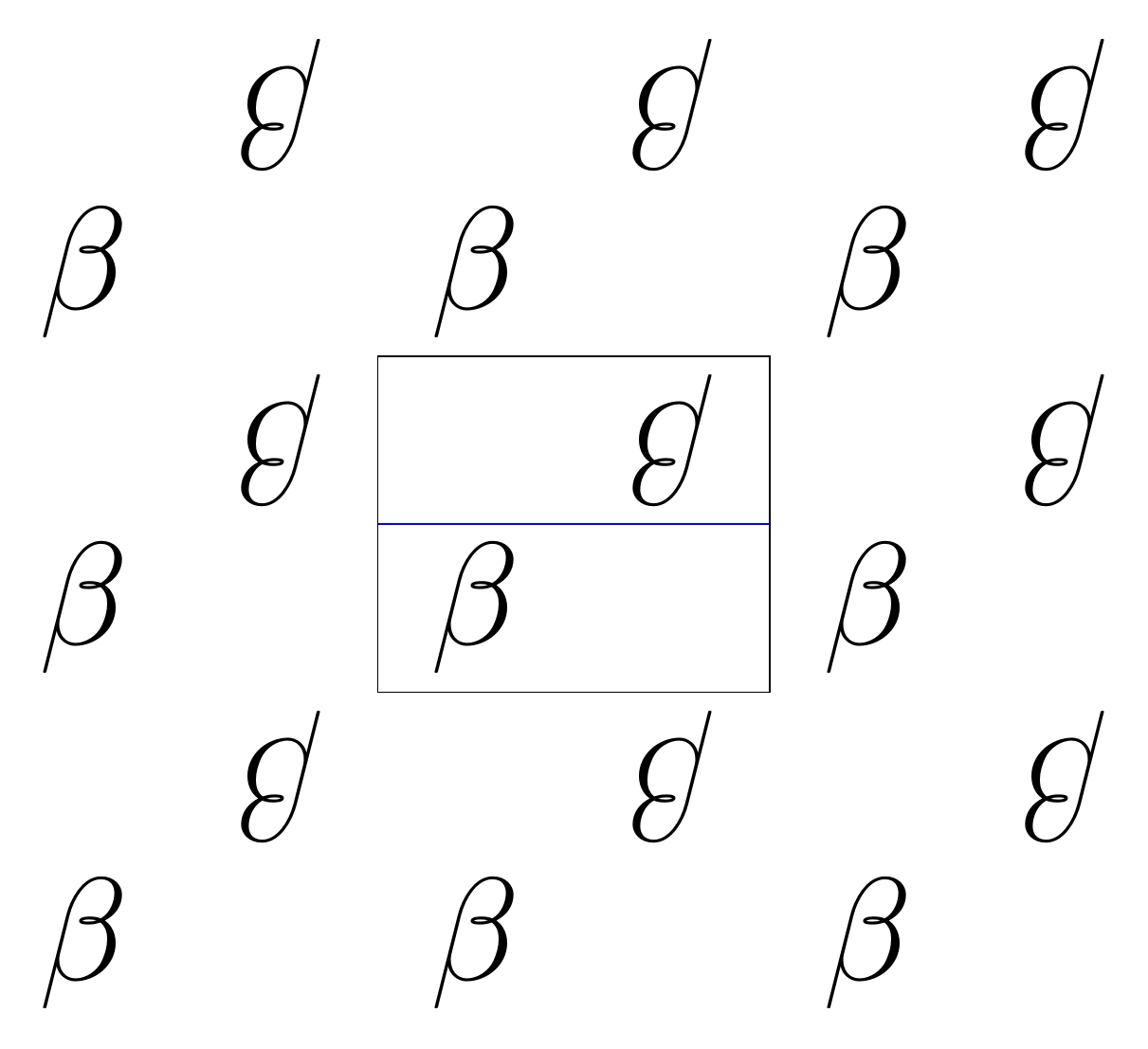}
    \caption{A visualisation of how the group $p2$ acts on $\R^2$. Each of the two rectangles with one blue edge is a fundamental domain of the group. A set of points forming  a letter $\beta$ is placed into one fundamental domain and acted on with the group.}
    \label{fig:pg-action}
\end{figure}

We say that $v\in \mathbb{R}^n$ is in \emph{special position} for a crystallographic group $\Gamma\leq  \E(n)$, if $\stab_{\Gamma}(v) \neq \{\Id\}$. If the stabiliser of $v$ is trivial, then $v$ is said to be in \emph{general position for $\Gamma$}.
For instance, the point $v=(1,1)^t\in \mathbb{R}^2$ forms a point in special position for the group $p2.$ This can be verified by the computation $v^{\phi \cdot \tau_1 \cdot \tau_1 \cdot \tau_2 \cdot \tau_2} ={(( -1,-1)^t)}^{\tau_1 \cdot \tau_1 \cdot \tau_2 \cdot \tau_2}=v$.
 Moreover, we define the \emph{translation (normal) subgroup} of $\Gamma$ as $\T(\Gamma) \coloneqq \{ (\phi_o, \phi_t) \in \Gamma \mid \phi_o = \Id \}$. The set $\L(\Gamma) \coloneqq \{\phi_t \mid (\phi_o, \phi_t) \in \T(\Gamma) \}$ is a full lattice of dimension $n$ satisfying $\T(\Gamma) \cong \L(\Gamma)$.
Additionally, the \emph{point group} of $\Gamma$ is defined as the quotient $\P(\Gamma) \coloneqq \Gamma / \T(\Gamma)$. 
If this quotient group is isomorphic to a subgroup of $\Gamma$, then we call $\Gamma$ \emph{symmomorphic}.  
 
 Previously we have introduced half spaces between different points in $\mathbb{R}^n$. Now, we use this definition to formally introduce Dirichlet domains.
\begin{definition}{\cite[Def. III.1]{plesken2014}}\label{def:dirichlet-cell}
	Let $O\subseteq \R^n$ be a discrete set and $u \in O$ a point. We denote the \emph{Dirichlet cell of $u$} by $D(u,O)$ and define it as the set
	$$
		D(u, O) \coloneqq \{ v \in \R^n \mid \forall w \in O \setminus \{ u \} : d(u, v) \leq d(v, w) \}=\bigcap_{w \in O, w \neq u} H^+(u, w).
	$$
\end{definition}

\section{Computing fundamental domains}
\label{section:comp}
The aim of this section is to present an algorithm that allows us to compute a fundamental domain for a given crystallographic group. 
Our algorithm requires as input a point $u\in \R^n$ in general position for a given crystallographic group. This point is then used to construct a Dirichlet cell forming a fundamental domain for $\Gamma.$
The details can be found in Algorithm~\ref{alg:dirichletcell}. In order to achieve this goal, we need some preparation.  
We start by examining points in special and general position for $\Gamma$.
First, since a fundamental domain $F$ contains a system of representatives $V$ satisfying $F^{\circ}\subseteq V \subseteq F$, we obtain the following lemma as a direct consequence.

\begin{lemma}\label{lma:translation-moves-interior}
	Let $\Gamma \leq \E(n)$ be a crystallographic group, $F$ a fundamental domain for $\Gamma$ and $v \in F^\circ$. If a non-trivial $\alpha \in \Gamma$ satisfies $v^\alpha \neq v$, then $v^\alpha \notin F^\circ$.
\end{lemma}

\begin{proposition}\label{prop:special-pos-inside-fund-dom}
	Let $\Gamma\leq \E(n)$ be a crystallographic group, $F \subset \R^n$ a fundamental domain for $\Gamma$ and $v\in F$. If $v$ is a point in special position for $\Gamma$, then $v \in \partial F$.
\end{proposition}
\begin{proof}
We prove the above statement by contradiction and therefore assume that $v$ is a point in special position satisfying $v\in F^{\circ}.$
	These two conditions imply that there exists a non-trivial $\alpha=(\alpha_o,\alpha_t) \in \stab_\Gamma(v)$ and $\delta \in \R_{>0}$ such that $v^\alpha=v$ and $B_\delta(v) \subset F^\circ$.
	We claim that there is a point $w \in B_\delta(v)$ such that $w^\alpha \neq w$. In order to show this, we assume that $w^\alpha = w$ holds for all $w \in B_\delta(v)$. Thus, for $1\leq i \leq n$ we define the point $w_i$ as $w_i \coloneqq v + \frac{\delta}{2} e_i$, where $e_i$ is the $i$-th standard basis vector.
	Since $d(w_i, v) < \delta$ for all $1 \leq i \leq n,$ the points $w_1,\ldots,w_n$ are all contained in $B_\delta(v)$. That means $w_i^\alpha = w_i$ by our assumption. 
    Thus, we obtain
 \begin{align*}
        v+\tfrac{\delta}{2}e_i &= w_i = w_i^\alpha = \left(v+\tfrac{\delta}{2}e_i \right) \cdot \alpha_o + \alpha_t = (v\cdot{\alpha_o} + \alpha_t)  + \tfrac{\delta}{2}e_i\cdot \alpha_o =v^\alpha + \tfrac{\delta}{2}e_i\cdot \alpha_o = v +\tfrac{\delta}{2}e_i\cdot \alpha_o 
    \end{align*}
     for all $1 \leq i \leq n$. Thus, $e_i=e_i\cdot \alpha_o.$
    Since $e_1,\ldots,e_n$ is a basis of $\mathbb{R}^n$, the orthogonal component of $\alpha$ satisfies $\alpha_o = \Id$. This implies $\alpha \in \T(\Gamma)$.
	Now, $\alpha$ is a non-trivial translation, hence we get $w^\alpha\neq w$ and therefore $w^\alpha \notin F^\circ$ by Lemma~\ref{lma:translation-moves-interior}. This results in $w^\alpha \notin B_{\delta}(v)$ which contradicts $w^\alpha = w$ for all $w \in B_\delta(v)$. 
	Thus, there exists a $w \in B_\delta(v) \setminus \{ v \}$ such that $w^\alpha \neq w$.
	Since $\alpha$ is an isometry of $\R^n,$ we conclude that $d(v, w^\alpha)=  d(v^\alpha, w^\alpha)  =d(v,w) < \delta$ i.e.\ $w^\alpha \in B_{\delta}(v) \subset F^\circ$ holds.
	Hence, two elements in the $\Gamma$-orbit of $w$ are contained in $F^\circ$, namely $w^\alpha$ and $w$. This contradicts $F$ being a fundamental domain for $\Gamma$.\\
\end{proof}

Next, we establish a correspondence between the Dirichlet cells of points in general position that lie in the same $\Gamma$-orbit.
\begin{lemma}\label{lma:dirichlet-identity}
	Let $\Gamma \leq \E(n)$ be a crystallographic group. If $u \in \R^n$ is a point in general position for $\Gamma$, then $D(u^\gamma, u^\Gamma) = D(u, u^\Gamma)^\gamma$ for all $\gamma \in \Gamma$.
\end{lemma}
\begin{proof}
Let $\gamma \in \Gamma$ be an isometry of $\mathbb{R}^n$.
We start this proof by showing $D(u^\gamma,u^\Gamma)\subseteq D(u,u^\Gamma)^\gamma.$ Thus, let $v\in \mathbb{R}^n$ be a point in $ D(u^\gamma, u^\Gamma)$. That means that $d(u^\gamma, v) \leq d(w,v)$ for all $w \in u^\Gamma$.
	We define $\tilde{v} \coloneqq v^{\gamma^{-1}}$. Since $\gamma$ is an isometry, we obtain 
	\begin{align*}
	 d(u, \tilde{v}) = d(u^\gamma, \tilde{v}^\gamma)= d(u^\gamma, (v^{\gamma^{-1}})^\gamma) = d(u^\gamma, v)
		\leq  d(w, v) = d(w^{\gamma^{-1}}, v^{\gamma^{-1}}) = d(w^{\gamma^{-1}}, \tilde{v}).
	\end{align*}
    for all $w \in u^\Gamma$. Note, the above inequality holds for all $w^{\gamma^{-1}} \in u^\Gamma$. Thus $\tilde{v} \in D(u, u^\Gamma)$ which implies $v = (v^{\gamma^{-1}})^\gamma \in D(u, u^\Gamma)^\gamma$.
Now, let $v'$ be a point in $ D(u, u^\Gamma)^\gamma$. Hence, there is a $v \in D(u, u^\Gamma)$ such that $v^\gamma = v'$. As described in the above case, we can use the isometry $\gamma$ to conclude:
	$$
d(u^\gamma, v') = d(u^\gamma, v^\gamma) = d(u, v) \leq d(w, v) = d(w^\gamma, v^\gamma) = d(w^\gamma, v')
	$$
for all $w \in u^\Gamma$. 
	Since the above inequality holds for all $w^\gamma \in u^\Gamma$, we get $v' \in D(u^\gamma, u^\Gamma)$.
\end{proof}

Now, we are able to state \Cref{thm:dirichlet-is-fund-dom}, which forms the basis of our proposed algorithm. This theorem establishes that the Dirichlet cell that is constructed for a point in general position for a crystallographic group forms a fundamental domain for the given group. This result can already be found in \cite{plesken1994}. In Appendix \ref{apx:thm:dirichlet-is-fund-dom}, we give a proof of this statement for completeness.

\begin{theorem}{\cite[Thm. III.11 (ii)]{plesken1994}}\label{thm:dirichlet-is-fund-dom}
	Let $\Gamma \leq \E(n)$ be a crystallographic group and $u \in \R^n$ a point in general position for $\Gamma$.
	Then the Dirichlet cell $D(u, u^\Gamma)$ is a fundamental domain for $\Gamma$.
\end{theorem}

With Theorem \ref{thm:dirichlet-is-fund-dom} in place, our approach for computing Dirichlet cells in the given setting is now justified. To make practical use of this result, we have to derive a formulation of this statement that is computationally applicable. First, we have to argue that only finitely many points in $u^\Gamma$ are required to construct $D(u,u^\Gamma)$ if $u$ is in general position for $\Gamma\leq \E(n)$. Additionally, we have to show that these points arise from the action of elements in $\Gamma$ that can be written as words of bounded length with respect to an appropriate generating set. The choice of such a generating set is detailed below.
\begin{remark}\label{not:cryst-grp-cosets}
Let $\Gamma \leq \E(n)$ be a crystallographic group. We know that there exist index sets $I,K$ and a generating set $S=\{\rho_i\mid i \in I\}\cup \{\tau_k\mid k \in K\}$ of $\Gamma$ such that 
\begin{enumerate}[label=(\arabic*)]
    \item $\tau_k \in \T(\Gamma)$ for all $k \in K$,
    \item $\{ (\tau_k)_t \mid k \in K \}$ forms a basis of the lattice $\mathcal{L}(\Gamma)$ and
    \item $\{ \rho_i\mid i\in I\}$ is a set of representatives for the cosets of $\T(\Gamma)$ in $\Gamma$, whence 
        $$
            \Gamma = \bigcup_{i \in I} \rho_i \T(\Gamma).
        $$
\end{enumerate}
 
From now on, we rely heavily on a generating set $S$ of $\Gamma$ being chosen to satisfy these conditions. Note, since $\mathcal{L}(\Gamma)$ is a full lattice of dimension $n$, the set $K$ can be chosen as $K:=\{1,\ldots,n\}$. If $\Gamma$ is a symmomorphic crystallographic group we choose the $\rho_i$ such that they form a generating set of $\P(\Gamma)$.
\end{remark}

With a generating set $S$ as described above, we can show that the translations in $\Gamma$ induced by points with a bounded norm can be written as words in $S$ of bounded length. The length can be bounded by vectors related to the lattice basis.

\begin{remark}
Let $\Gamma\leq \E(n)$ be a crystallographic group and $S=\{\rho_i\mid i \in I\}\cup \{\tau_k\mid k \in K\}$ a generating set of $\Gamma$ as described in \Cref{not:cryst-grp-cosets}. Let $T\in \R^{n\times n}$ denote the \emph{lattice matrix} of ${\mathcal L}(\Gamma)$, that is the matrix whose $k$-th column is ${(\tau_k})_t$ for $k\in \{1,\ldots, n\}$. The reciprocal lattice 
${\mathcal L}(\Gamma)^\ast$ consists of those vectors $w\in \R^n$ for which $w\cdot v \in \Z$ for all $v\in {\mathcal L}(\Gamma)$ and has a lattice matrix $T^{-t}.$ We call the columns of $T^{-t}$ the \emph{reciprocal basis vectors} of the basis $T$. 
 \end{remark}

If we define $\Vert C\Vert_{2,1} := \sum_{i=1}^n \Vert c_i\Vert$ for a matrix $C\in \R^{n\times n}$ with rows $c_1, \ldots, c_n,$ then we obtain the following result.

\begin{lemma}\label{lma:translation-short-word}
Let $\Gamma\leq \E(n)$ be a crystallographic group and $S=\{\rho_i\mid i \in I\}\cup \{\tau_k\mid k \in K\}$ a generating set of $\Gamma$ as described in \Cref{not:cryst-grp-cosets}. Furthermore, let $\delta\in \R_{>0}$ be a positive real number.
	Then there exists an $A(\delta) \in \N $ such that every $\tau \in \T(\Gamma)$ satisfying $\Vert(\tau)_t\Vert \leq \delta$ can be expressed as a word in $\{ \tau_k \mid k \in K \}$ of length at most $A(\delta)$. Moreover, $A(\delta) \le  \Vert B \Vert_{2,1} \cdot \delta$, where $B\in \R^{n\times n}$ is the 
     corresponding reciprocal matrix of $\mathcal{L}(\Gamma)$. 
\end{lemma}
\begin{proof}
By \Cref{not:cryst-grp-cosets}, $\{\tau_k\mid k \in K\}$ generates the translation normal subgroup $\mathcal{T}(\Gamma)$. Further, we know that $\mathcal{T}(\Gamma)$ satisfies $\T(\Gamma) \cong \L(\Gamma)$. Since the ball $B_\delta(0)$ contains only finitely many lattice points in $\mathcal{L}(\Gamma)$, the set $M_{\delta}:=\{\tau \in \mathcal{T}(\Gamma)\mid \Vert(\tau)_t\Vert \leq \delta\}$ is finite.
Furthermore, for every  $\tau \in M_\delta$, there exists a set of coefficients $\{a_k \in \Z \mid k \in K \}$ such that 
    $$
        \tau = \prod_{k \in K} \tau_k^{a_k}.
    $$ 
Hence, every $\tau\in M_{\delta}$ can be written as a word in $\{\tau_k\mid k \in K\}$ of length at most $A(\delta),$ where $A(\delta)$ is defined as
\begin{equation}\label{eq:A(delta)}
    A(\delta) \coloneqq \max \left\{ \sum_{k \in K} | a_k| \, \middle| \, \tau = \prod_{k \in K} {\tau_k}^{a_k} \in M_{\delta} \right\}.
\end{equation}
For $\tau \in M_\delta$ with $\tau = \prod_{k \in K} \tau_k^{a_k}$ define $a := (a_1, \ldots, a_n)^T$. Then $a = B \cdot (\tau)_t$ where $B$ is the reciprocal lattice matrix. If $b_i$ denotes the $i$-th row of $B$, we find by the Cauchy-Schwarz inequality:
    \[ \Vert a\Vert_1 = \sum_{k=1}^n |a_k| \le \sum_{k=1}^n \Vert b_k\Vert \cdot \Vert \tau_t \Vert = \Vert B\Vert_{2,1} \cdot \delta.\]
Therefore $A(\delta) \le \Vert B\Vert_{2,1} \cdot \delta. $
\end{proof}

 Now, we exploit this result and approach the word problem for arbitrary isometries. For this, we give the following remark which introduces a new notation and recalls an identity for the Dirichlet cell of a point in general position from  \cite[Satz III.2]{plesken1994}.

\begin{remark}\label{rmk:dirichlet-finite-intersection}
Let $\Gamma\leq \E(n)$ be a crystallographic group, $S=\{\rho_i\mid i \in I\}\cup \{\tau_k\mid k \in K\}$ a generating set for $\Gamma$ as described in \Cref{not:cryst-grp-cosets} and $u \in \R^n$ a point in general position for $\Gamma$. Since the fundamental domain $D(u, u^\Gamma)$ is a compact set in $\R^n$, there exists a real number $r\in \mathbb{R}_{>0}$ such that $D(u, u^\Gamma) \subseteq B_r(u)$.

For $d>0$  define 
$
W_d(u) \coloneqq \left( u^\Gamma \cap B_{d}(u) \right) \setminus \{ u \}.
$	
We obtain the following equality for $d=2r$:
$$
	D(u, u^\Gamma) = \bigcap_{w \in W_{2r}(u)} H^+(u, w).
$$
\end{remark}

Hence, if $r$ is given as in the above remark, then we have to find all elements in the set $W_{2r}(u)$ to compute a Dirichlet cell of a crystallographic group $\Gamma$. 

Note that $D(u, u^{\Gamma}) \subseteq D(u, u^{\T(\Gamma)})$ as $\T(\Gamma) \leq \Gamma$. The covering radius $r$ of a lattice is defined as the smallest radius $r>0$ such that balls of radius $r$ centred at all lattice points cover the entire space $\mathbb{R}^n$. Therefore, if $r$ is the covering radius of the lattice induced by $\T(\Gamma)$ it follows that $D(u, u^{\Gamma})$ is contained in a ball of radius $r$ around $u$.  
Combining this with the inequality from \cite[Thm. 7.9]{micciancio2002} we obtain the following upper bound:
$$
    r \leq \frac{\sqrt{n}}{2} \max_{k \in K} || \tau_k||.
$$

\begin{remark}\label{rem:epsilon}
Let $\Gamma\leq \E(n)$ be a symmomorphic crystallographic group, $S=\{\rho_i\mid i \in I\}\cup \{\tau_k\mid k \in K\}$ a generating set for $\Gamma$ as described in \Cref{not:cryst-grp-cosets} and $u \in \R^n$ a point in general position for $\Gamma$.
Furthermore, define $\epsilon \coloneqq \max_{i \in I} \Vert u^{\rho_i} - u\Vert$.
Then 
\[\epsilon =\max_{i \in I} \Vert u^{\rho_i} - u\Vert
= \max_{i \in I} \Vert {\rho_i} - Id\Vert \cdot \Vert u \Vert \le \max_{i \in I} (\Vert \rho_i\Vert + \Vert Id \Vert)\cdot \Vert u\Vert =  2 \cdot \Vert u \Vert,
\]
as the eigenvalues of $\rho_i$ are $\pm 1$, where $\Vert A\Vert$ for a matrix $A\in \mathbb{R}^{n\times n}$ is the matrix norm induced by the Euclidean norm on $\mathbb{R}^n$. 
\end{remark}

\begin{definition}\label{def:w-hat}
    Let $\Gamma\leq \E(n)$ be a crystallographic group, $S=\{\rho_i\mid i \in I\}\cup \{\tau_k\mid k \in K\}$ a generating set for $\Gamma$ as described in \Cref{not:cryst-grp-cosets} and $u \in \R^n$. For $\ell>0$ we define the set $\widehat{W}_{\ell}(u)$ as
	$$
		\widehat{W}_{\ell}(u) \coloneqq\{ u^\gamma \mid \gamma = \rho_i \tau , i \in I , \Vert \tau_t \Vert \leq \ell \}.
	$$
\end{definition}

Before concluding the main result, we need an additional lemma.

\begin{lemma}\label{lma:word-length-dist}
Let $\Gamma\leq \E(n)$ be a crystallographic group, $S=\{\rho_i\mid i \in I\}\cup \{\tau_k\mid k \in K\}$ a generating set for $\Gamma$ as described in \Cref{not:cryst-grp-cosets} and $u \in \R^n$ a point in general position for $\Gamma$.
Furthermore, let $d \in \R_{>0}$ be a positive real number and define $\epsilon \coloneqq \max_{i \in I} \Vert u^{\rho_i} - u\Vert$.
Then the set $W_d(u)$ (see \Cref{rmk:dirichlet-finite-intersection}) satisfies $W_d(u)  \subseteq \widehat{W}_{d+ \epsilon}(u)$.
\end{lemma}
\begin{proof}
	Let $w\in W_d(u)$, that is there exists an element $\gamma\in \Gamma$ such that $w = u^\gamma$ and $\Vert w - u \Vert < d$. Since $S$ is a generating set as described in \Cref{not:cryst-grp-cosets}, we can write $\gamma$ as $\gamma = \rho_i \tau$ for an $i \in I$ and $\tau \in \T(\Gamma)$. We  prove the above statement by contradiction. 
    Hence, we assume $w \notin \widehat{W}_{d+ \epsilon}(u)$, i.e.\ $\Vert\tau_t\Vert > d+\epsilon$. By using the inequality $\Vert \tau_t \Vert > d + \epsilon > \epsilon \geq  \Vert u^{\rho_i} - u \Vert$, we obtain
	\begin{align*}
		\Vert w - u \Vert 	 = \Vert u^{\rho_i} + \tau_t - u \Vert
					\geq \left| \Vert \tau_t \Vert - \Vert u^{\rho_i} - u\Vert
                    \right|
                    = \Vert \tau_t \Vert - \Vert u^{\rho_i} - u\Vert \geq \Vert \tau_t \Vert - \epsilon>d.
	\end{align*} 
    This implies $w\notin W_d(u)$,
	 contradicting our choice of $w$.
\end{proof}

By combining \Cref{lma:translation-short-word} and \Cref{lma:word-length-dist} we are able to establish the desired upper bound of the word length.

\begin{theorem}\label{thm:dirichlet-word-length}
Let $\Gamma\leq \E(n)$ be a crystallographic group, $S=\{\rho_i\mid i \in I\}\cup \{\tau_k\mid k \in K\}$ a generating set of $\Gamma$ as described in \Cref{not:cryst-grp-cosets} and $u \in \R^n$ a point in general position for $\Gamma$.
Further, let $r\in \mathbb{R}_{>0}$ such that $D(u, u^\Gamma) \subseteq B_r(u)$. Then the Dirichlet cell $D(u, u^\Gamma)$ can be computed by considering the intersection of half-spaces $H^+(u,w)$ for words $w$ of length at most $A(2r+\epsilon)+1$ with $\epsilon \coloneqq \max_{i \in I} \Vert u^{\rho_i} - u\Vert$.
Moreover if $\Gamma$ is symmomorphic, then $A(2r+\epsilon)+1  \le  2\cdot  \Vert B \Vert_{2,1} \cdot \frac{\sqrt{n}}{2} \max_{k \in K} || \tau_k|| + 1,$ where $B\in \R^{n\times n}$ is the corresponding reciprocal matrix of the lattice. 
\end{theorem}
\begin{proof}
	Consider $\widehat{W}_{d+\epsilon}(u)$ as defined as in \Cref{def:w-hat}, and define $d: = 2r$.
	By Lemma \ref{lma:word-length-dist} the Dirichlet cell around the point $u$ is equal to the intersection:
	$$
		D(u, u^\Gamma) = \bigcap_{w \in W_d(u)} H^+(u,w) = \bigcap_{w \in \widehat{W}_{d+\epsilon}(u)} H^+(u, w).
	$$
    
    By the choice of our generating set $S$, we know that $w \in \widehat{W}_{d+\epsilon}(u)$ satisfies $w = u^\gamma$ for an element $\gamma = \rho_i \tau$ with $i \in I$ and $\Vert\tau\Vert\leq d+\epsilon$. 

    Since $\tau$ is a word in $\{ \tau_k \mid k \in K \}$ with $|| (\tau)_t || < d+\epsilon $ (see \Cref{lma:translation-short-word}), it follows that $\gamma=\rho_i\tau$ is a word in $S$ of length at most $A(d+\epsilon)+1.$

    If $\Gamma$ is symmorphic we can additionally use Lemma~\ref{lma:translation-short-word} and Remark~\ref{rem:epsilon} to obtain the following bound:

    \[A(2r+\epsilon)+1 \le A(2(\rho  +\Vert u\Vert)) + 1 \le 2\cdot  \Vert B \Vert_{2,1} \cdot \frac{\sqrt{n}}{2} \max_{k \in K} || \tau_k|| + 1.
    \]
    
\end{proof}

Hence, for a symmomorphic crystallographic group and a corresponding point in general position as described in the statement of the theorem above, it suffices to consider words of length at most $ 2\cdot  \Vert B \Vert_{2,1} \cdot \frac{\sqrt{n}}{2} \max_{k \in K} || \tau_k|| + 1$ in order to compute the corresponding Dirichlet cell forming a fundamental domain.
In the examples we have computed, however, the resulting upper bound $ 2\cdot  \Vert B \Vert_{2,1} \cdot \frac{\sqrt{n}}{2} \max_{k \in K} || \tau_k|| + 1$ is often larger than necessary. For instance, if we assume that for a given symmomorphic crystallographic group $\Gamma$ the point $u=0$ is in general position and the basis in dimension $n=3$ is the standard basis, then we obtain $ 2\cdot  \Vert B \Vert_{2,1} \cdot \frac{\sqrt{n}}{2} \max_{k \in K} || \tau_k|| + 1 \leq 7$. In all computed examples, no word length greater than $5$ was required, and in many cases a length of $4$ already sufficed. If e.g. the size of of a given generating set is $4$ we get a set of element to consider of size $4^7 = 16384$, however in practice considering $4^4 = 256$ elements was enough.

Thus, Dirichlet cells forming fundamental domains of $\Gamma$ can be computed by considering elements in $\Gamma$ expressible as words of bounded length in a suitable generating set. 

In order to optimise our proposed algorithm for the computation of fundamental domains, we add another ingredient. This is based on the following observation: 
Let $\Gamma\leq \E(n), S$ and $u$ be given as in \Cref{not:cryst-grp-cosets} and  $\epsilon \coloneqq \max_{i \in I} \Vert u^{\rho_i} - u\Vert$ as described in \Cref{lma:word-length-dist}. We then choose a positive real number $r_1$ such that the set 
$$
	D_1 = \bigcap_{w \in \widehat{W}_{2r_1}(\epsilon)} H^+(u, w).
$$ 
is a convex polytope in $\mathbb{R}^n$, see \Cref{lma:word-length-dist} for the definition of the set $\widehat{W}_{2r_1}(\epsilon)$. It is easy to see that $D_1$ satisfies $D(u,u^\Gamma)\subseteq D_1.$ 
However, in order to know if this inclusion is already an equality we give a volume argument. 
In particular, we will see in the following that every fundamental domain for $\Gamma$ has the same volume (\Cref{thm:fund-dom-vol}). 
Here, the \emph{volume} $\vol(B)$ of a set $B\subseteq \R^n$ is defined as the Lebesgue measure of $B$.
Hence, if $D_1$ satisfies $\vol(D(u,u^\Gamma))< \vol(D_1)$, we have to choose a positive real number $r_2>r_1$ and construct the polytope $D_2$ given by
$$
	D_2 \coloneqq \bigcap_{w \in \widehat{W}_{2r_2}(\epsilon)} H^+(u, w).
$$
This process can be iterated. Hence, we will obtain radii $r_1< r_2 < r_3< \ldots$ such that for every $i \in \mathbb{N}$ the positive number $r_i$ gives rise to the polytope  
$$
	D_i = \bigcap_{w \in \widehat{W}_{2r_i}(\epsilon)} H^+(u, w).
$$

If the incrementation of the radii is chosen carefully ($r_{i+1}>r_i+1$ for instance), we will obtain a smallest index $m$ such that $\vol(D_m)=\vol(D(u,u^\Gamma)).$ Since $D_m$ and $D(u,u^\Gamma)$ are compact sets with $D(u,u^\Gamma)\subseteq D_m$, this implies $D_m=D(u,u^\Gamma)$ and we have therefore computed our fundamental domain.

The remainder of this section is devoted to establishing \Cref{thm:fund-dom-vol} and giving a pseudocode describing our algorithm.
First, it is easy to see that applying isometries to subsets of $\mathbb{R}^n$ leaves the corresponding volumes invariant.
\begin{proposition}\label{thm:vol-invar-under-isom}
	Let $B \subset \R^n$ be a closed subset. If $\phi \in \E(n)$ is an isometry, then $\vol(B^\phi) = \vol(B)$.
\end{proposition}

Next, we establish a relation between fundamental domains of a crystallographic group $\Gamma$ and fundamental domains of the translation normal subgroup $\mathcal{T}(\Gamma).$
\begin{proposition}\label{prop:translation-cell-fund-dom}
    Let $\Gamma\leq \E(n)$ be a crystallographic group and $S=\{\rho_i\mid i \in I\}\cup \{\tau_k\mid k \in K\}$ a generating set of $\Gamma$ as described in \Cref{not:cryst-grp-cosets}. If $F$ forms a fundamental domain for $\Gamma$, then
    $$
		C \coloneqq \bigcup_{i \in I} F^{\rho_i}
	$$
	is a fundamental domain for $\T(\Gamma) \leq \E(n)$. 
\end{proposition}
Note, we call the set $C$ in the above statement a \emph{translation cell} of $\Gamma$.
In \ref{apx:prop:translation-cell-fund-dom} we provide a proof of this statement in the appendix and proceed by computing the volume of a fundamental domain.

\begin{theorem}\label{thm:fund-dom-vol}
Let $\Gamma\leq \E(n)$ be a crystallographic group and $F$ a fundamental domain for $\Gamma$. If $C$ forms a translation cell of $\Gamma$ satisfying $F \subseteq C$, then 
	$$\vol(F) = \frac{\vol(C)}{|\P(\Gamma)|}.$$
\end{theorem}
\begin{proof}
We know that a translation cell is a fundamental domain for $\T(\Gamma)$ (see Proposition~\ref{prop:translation-cell-fund-dom}). We observe that $C$ contains $|I|$ copies of the fundamental  domain $F$. Since these copies are intersection free sets that lie inside $C$, we obtain the above equality.
\end{proof}

With the above result, we are able to reduce the computation of the volume of a fundamental domain for $\Gamma$ to computing the volume of the translation cell without computing the fundamental domain itself. Since the volume of the translation cell of $\Gamma$ is the determinant of the lattice $\mathcal{L}(\Gamma)$, obtaining the desired volume is straight forward.

By combining this observation with \Cref{thm:dirichlet-word-length} we are able to give our proposed algorithm.

\begin{algorithm}[H]
	\caption{Dirichlet Cell}\label{alg:dirichletcell}
	\LinesNumberedHidden
	\SetAlgoLined
	\KwData{a crystallographic group $\Gamma \leq \E(n)$ given by a set of generators $\{ \rho_i, \tau_k \mid i \in I, k \in K \}$ as described in Remark~\ref{not:cryst-grp-cosets} and a point $\texttt{u}$ in general position for $\Gamma$}
	\KwResult{\texttt{poly}, a polyhedron that is a fundamental domain. }

    \hrulefill
    
    $\texttt{fdVol} \gets $ Volume of fundamental domain computed with Theorem~\ref{thm:fund-dom-vol}\;
    $\texttt{words} \gets \{ \rho_i, \tau_k \mid i \in I, k \in K \}$

    \Comment{Compute initial candidate}
    $\texttt{orbit} \gets [\,]$\;
	\For{$\gamma$ in $\texttt{words}$}{
		Add($\texttt{orbit}$, $u^\gamma$)\;
	}

    $\texttt{halfspaces} \gets$ halfspaces $H^+(u,v) $ for all $v$ in $\texttt{orbit}$\;
    $\texttt{fdCand} \gets$ polyhedron given by intersection of $\texttt{halfspaces}$\;

    \Comment{Consider longer words until right volume is reached}
    \While{$vol(\texttt{fdCand}) > \texttt{fdVol}$}{
        $\texttt{words} \gets [\texttt{words}, \alpha \cdot \beta \text{ for }\alpha $ in $ \texttt{words},  \beta $ in $\{ \rho_i, \tau_k \mid i \in I, k \in K \}$\;

    	\For{ $\gamma \in \texttt{words}$}{
    		Add($\texttt{orbit}$, $u^\gamma$)\;
    	}

        $\texttt{halfspaces} \gets$ half-spaces $H^+(u,v)$ for all $v$ in $\texttt{orbit}$\;
        $\texttt{fdCand} \gets$ polyhedron given by intersection of $\texttt{halfspaces}$\;
    }

	return $\texttt{fdCand}$\;
\end{algorithm}

Algorithm \ref{alg:dirichletcell} is implemented in the computer algebra system \emph{GAP} \cite{GAP4} and is available via the package \emph{CrystFundDom} \cite{crystfunddom}, with additional visualisations available through \emph{GAPic} \cite{gapic}. 

\section{Constructing Topological Interlocking Assemblies}
\label{section:TIA}

In this section, we explore potential applications of 3-dimensional crystallographic groups. For this purpose, we make use of our results in \Cref{section:comp}. Our aim is to exploit a given 3-dimensional crystallographic group and construct a corresponding fundamental domain such that copies of this fundamental domain can be arranged in Euclidean 3-space into a so-called \emph{topological interlocking assembly (TIA)}.
In order to give a brief description of a topological interlocking assembly, we have to recall some notions from~\cite{goertzen2024constructinginterlockingassembliescrystallographic}. 
The definition of a block is a non-empty subset $X \subset \mathbb{R}^3$ that is connected, compact and further satisfies $X= \overline{{X^\circ}}$. An assembly $M$ is a countable set of blocks such that any two distinct blocks $X,X'\in M$ satisfy $X\cap X'= \partial X\cap \partial X'$. A frame $J$ of a given assembly $M$ is a proper non-empty subset of $M$.
Furthermore, a motion is a continuous map $\phi:[0,1]\to \SE(3)$ with $\phi(0)=Id$, where $\SE(3)$ denotes the group of Euclidean motions. 
Then {(loosely speaking)}, a topological interlocking assembly (TIA) is an assembly of blocks constrained by a fixed frame $J$ such that every finite non-empty subset of blocks cannot be moved by continuous motions without causing intersections. In the literature, much attention has been given to investigations involving multiple copies of the same block, see \cite{GoertzenPhD,osteomorphic-block-estrin}.
As an example, Dyskin et.\ al have utilised the Platonic solids to construct TIAs and have further explored the mechanical properties of the resulting assemblies \cite{topological-interlocking-platonic-solids}. 

 The study of TIAs dates back to the 18th century. Early studies on TIAs can be found in the works of Abeille in 1735 \cite{abeille_memoire_1735} and Frezier in 1738 \cite{frezier_theorie_1738}.
Since the impactful work of Dyskin et al.\ \cite{dyskin_new_2001}, who have reintroduced TIAs as a material design concept, planar TIAs have been studied extensively, see \cite{VoroNoodles,goertzen2025influence,kanel-belov_interlocking_2010,Wang-2019-Topolock}. Although TIAs are widely used in civil engineering and architecture, there has been less investigation of these assemblies from a purely mathematical perspective. The first (mathematically) exact definition of a TIA was given in \cite{goertzen2022topological}.  T.\ Goertzen 
\cite{goertzen2024constructinginterlockingassembliescrystallographic,goertzen2024mathematicalfoundationsinterlockingassemblies}
explored TIAs that have been constructed by exploiting 2-dimensional crystallographic groups. 
Inspired by these works, we explore the construction of TIAs by employing 3-dimensional crystallographic groups. Note, another approach to construct three dimensional TIAs has been presented in \cite{EBERT2024103779} which constructs handlebody plesiohedra. Our approach can be summarised as follows.

\begin{approach}
Let $\Gamma\leq \E(3)$ be a crystallographic group and $u\in \mathbb{R}^3$ a point in general position for $\Gamma$. Our aim is to construct a block $X$ and find $\phi_i\in \Gamma,$ where $i$ is an element of a countable index set $I$ such that $\{X^{\phi_i}\mid i \in I\}$ forms a TIA. Here, the construction of $X$ is achieved as follows: 
\begin{enumerate}
    \item We use the point $u$ and Algorithm~\ref{alg:dirichletcell} to construct a Dirichlet cell $D$ that forms a fundamental domain for $\Gamma,$
    \item Next, we use $D$ to obtain another fundamental domain $F$ of $\Gamma$ that does not form a convex set in $\R^3.$ This is further illustrated in \Cref{def:compatible-deformation}.
\end{enumerate}

\end{approach}

Note, the group elements $\phi_i$ with $i\in I$ can be chosen depending on the application of the desired TIA. In the following we will consider the crystallographic group $p23$ and give two different sets of group elements to construct candidates of TIAs: one realising a tripod (see \Cref{fig:ass-195}) and another one approximating the Stanford Bunny (see \Cref{fig:p23-bunny}). 

Step 2 of our described approach requires that a convex fundamental domain for a given crystallographic group can be transformed into a non-convex fundamental domain for the same group. This is achieved by the deformation of a fundamental domain as described below. 

\begin{definition}\label{def:compatible-deformation}
    Let $\Gamma \leq E(3)$ be a crystallographic group with fundamental domain $F$ and $D\subset \R^3$ a closed set such that $\partial F \cap \partial D \neq \emptyset$. We define the set $G$ as 
    $$
        G \coloneqq \overline{(F \cup D) \setminus D^{\Gamma \setminus \Id}}.
    $$
    If $G$ is a fundamental domain for $\Gamma$, then we call $G$ the \emph{deformation of $F$ by $D$}.
\end{definition}

In the following, we give an example of a fundamental domain obtained from \Cref{def:compatible-deformation}.
In particular, we illustrate our method to construct assemblies that have a high potential of being TIAs by using the crystallographic group $p23\leq \E(3)$.
In order to define $p23$, let $\rho_i,\tau_i$ for $i=1,2,3$ be isometries in $\E(3)$ given by
\begin{align*}
&\rho_1 \coloneqq \left( \begin{pmatrix}
	0& 1& 0\\
	0& 0& 1\\ 
	1& 0& 0
\end{pmatrix}, \begin{pmatrix} 0\\ 0\\ 0 \end{pmatrix} \right),
&&\rho_2 \coloneqq \left( \begin{pmatrix}
	-1& 0& 0\\
	0& 1& 0\\ 
	0& 0& -1
\end{pmatrix}, \begin{pmatrix} 0\\ 0\\ 0 \end{pmatrix} \right),\\
&\rho_3 \coloneqq \left( \begin{pmatrix}
	-1& 0& 0\\
	0& -1& 0\\ 
	0& 0& 1
\end{pmatrix}, \begin{pmatrix} 0\\ 0\\ 0 \end{pmatrix} \right),
&&\tau_1 \coloneqq \left( \begin{pmatrix}
	1& 0& 0\\
	0& 1& 0\\ 
	0& 0& 1
\end{pmatrix}, \begin{pmatrix} 1\\ 0\\ 0 \end{pmatrix} \right),\\
&\tau_2 \coloneqq \left( \begin{pmatrix}
	1& 0& 0\\
	0& 1& 0\\ 
	0& 0& 1
\end{pmatrix}, \begin{pmatrix} 0\\ 1\\ 0 \end{pmatrix} \right),
&&\tau_3 \coloneqq \left( \begin{pmatrix}
	1& 0& 0\\
	0& 1& 0\\ 
	0& 0& 1
\end{pmatrix}, \begin{pmatrix} 0\\ 0\\ 1 \end{pmatrix}\right).
\end{align*}
These elements generate the group $p23$, hence $
	p23 \coloneqq \langle \rho_1, \rho_2, \rho_3, \tau_1, \tau_2, \tau_3 \rangle.
$
If we consider the point $u \coloneqq ( \frac{1}{4} \; \frac{1}{8} \;\frac{1}{4})^t\in \R^3$ and apply Algorithm~\ref{alg:dirichletcell} with the group $p_{23}$ given by the above generators as input, we obtain a fundamental domain $D_{23}$ given by the convex hull of the following points
$$
\left\{
\begin{pmatrix} -\frac{1}{2} \\ -\frac{1}{2} \\ -\frac{1}{2} \end{pmatrix},
\begin{pmatrix} -\frac{1}{3} \\ -\frac{1}{3} \\ -\frac{2}{3} \end{pmatrix},
\begin{pmatrix} 0 \\ 0\\ -\frac{1}{2} \end{pmatrix},
\begin{pmatrix} 0 \\ 0 \\ 0 \end{pmatrix},
\begin{pmatrix} -\frac{1}{2} \\ 0 \\ 0 \end{pmatrix},
\begin{pmatrix} -\frac{1}{3} \\ \frac{1}{3} \\ -\frac{1}{3} \end{pmatrix},
\begin{pmatrix} -\frac{1}{2} \\ 0 \\ -\frac{1}{2} \end{pmatrix},
\begin{pmatrix} -\frac{2}{3} \\ -\frac{1}{3} \\ -\frac{1}{3} \end{pmatrix},
\begin{pmatrix} -\frac{1}{4} \\ -\frac{1}{4} \\ -\frac{1}{4} \end{pmatrix}
\right\}.
$$

A visualisation of the fundamental domain $D_{23}$ in three different views can be found in Figure~\ref{fig:initial-fund-dom-195}.

\begin{figure}[H]
    \centering
    \hfill
    \begin{subfigure}[b]{0.3\textwidth}
        \includegraphics[width=\textwidth]{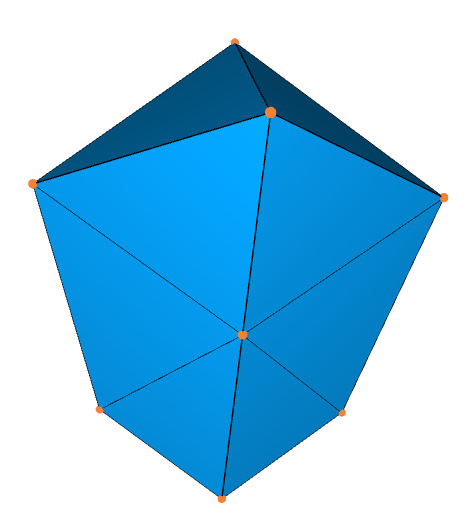}
    \end{subfigure}
    \hfill
    \begin{subfigure}[b]{0.3\textwidth}
        \includegraphics[width=\textwidth]{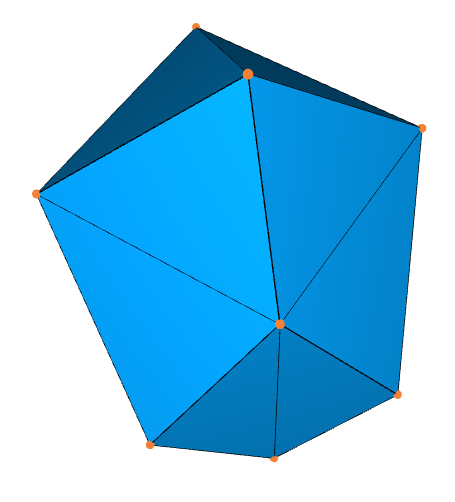}
    \end{subfigure}
    \hfill
    \begin{subfigure}[b]{0.3\textwidth}
        \includegraphics[width=\textwidth]{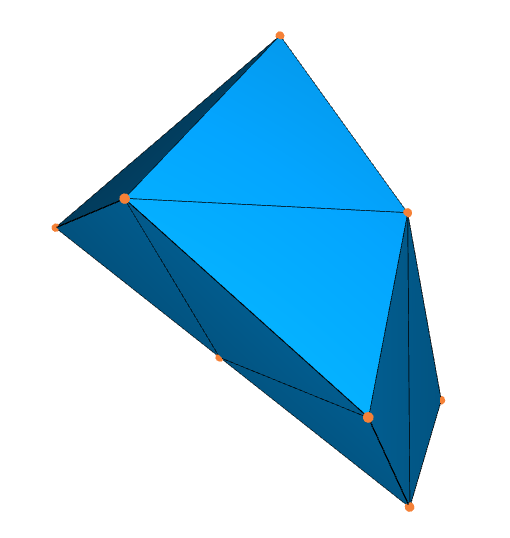}
    \end{subfigure}
    \hfill
    \caption{Different views of the fundamental domain $D_{23}$ of $p23$.}
    \label{fig:initial-fund-dom-195}
\end{figure}

Next, we deform $D_{23}$ to obtain a non-convex fundamental domain $F_{23}$. As described in \Cref{def:compatible-deformation} we can obtain a new fundamental domain for $p23$ by applying a deformation to $D_{23}.$ 
The resulting fundamental domain $F_{23}$ of $p23$ is illustrated in the below figure.
\begin{figure}[H]
    \centering
    \hfill
    \begin{subfigure}[b]{0.3\textwidth}
        \includegraphics[width=\textwidth]{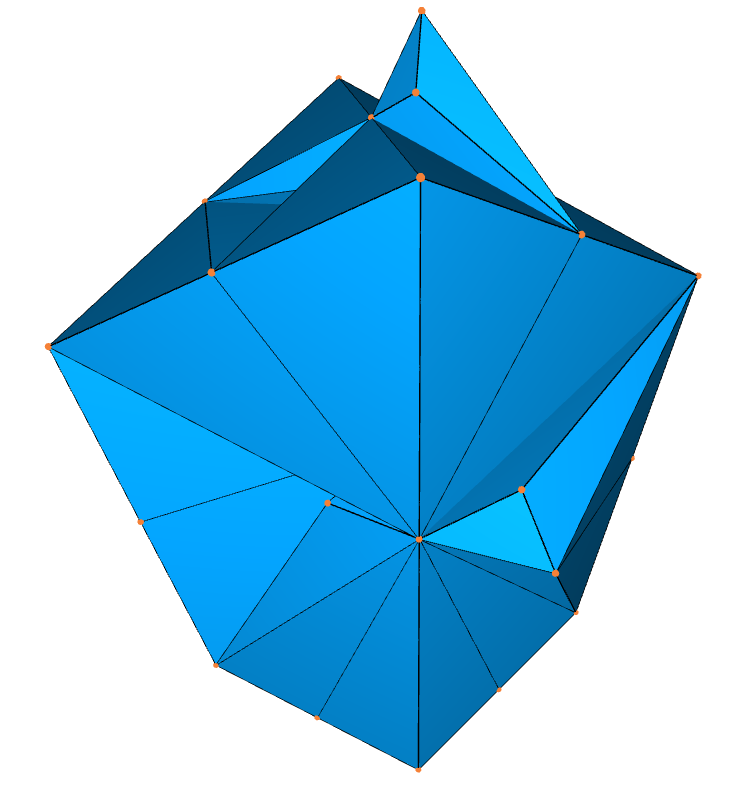}
    \end{subfigure}
    \hfill
    \begin{subfigure}[b]{0.3\textwidth}
        \includegraphics[width=\textwidth]{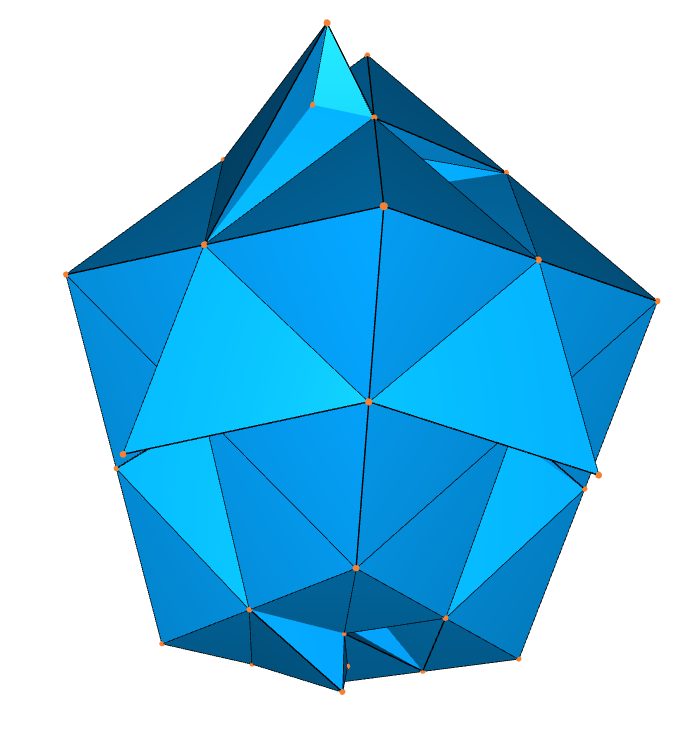}
    \end{subfigure}
    \hfill
    \begin{subfigure}[b]{0.25\textwidth}
        \includegraphics[width=\textwidth]{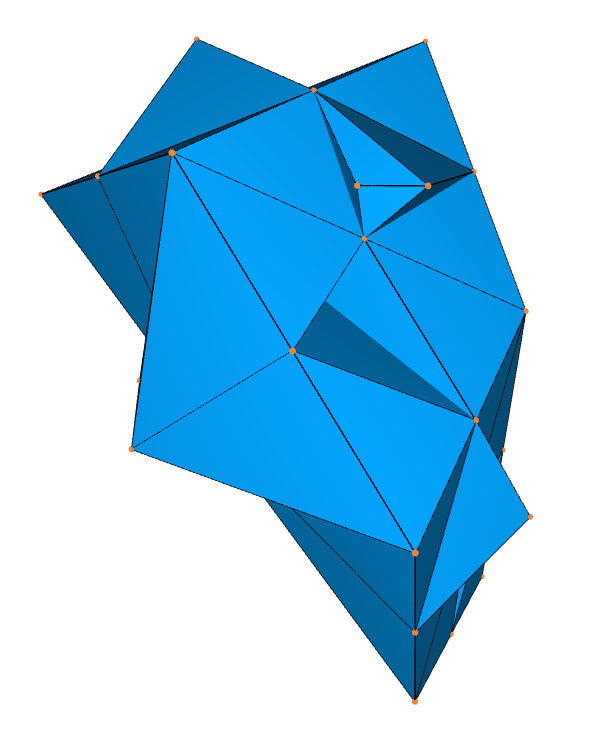}
    \end{subfigure}
    \hfill
    \caption{Different views of the deformed fundamental domain $F_{23}$ of $p23$.}
    \label{fig:def-fund-dom-195}
\end{figure}

Since $F_{23}$ is a fundamental domain for $p23$, and hence space-filling, we can create assemblies consisting of copies of $F_{23}$ by exploiting the group $p23$ as described in \Cref{def:compatible-deformation}. By only choosing a subset of the blocks we can approximate arbitrary surfaces.
To demonstrate this, we approximated the Stanford bunny (\cite{stanford_bunny_stl}). This approximation is achieved with $1323$ blocks forming copies of $F_{23}$. Note, the approximation could be refined by considering lower scaled fundamental domains and thus more blocks to approximate the shape. The resulting assembly is shown in Figure~\ref{fig:p23-bunny}.

\begin{figure}[H]
    \centering
    \newcommand{\imgboxheight}{4cm} 

    \begin{subfigure}[b]{0.22\textwidth}
        \begin{minipage}[t][\imgboxheight][t]{\textwidth}
            \centering
            \includegraphics[width=\textwidth]{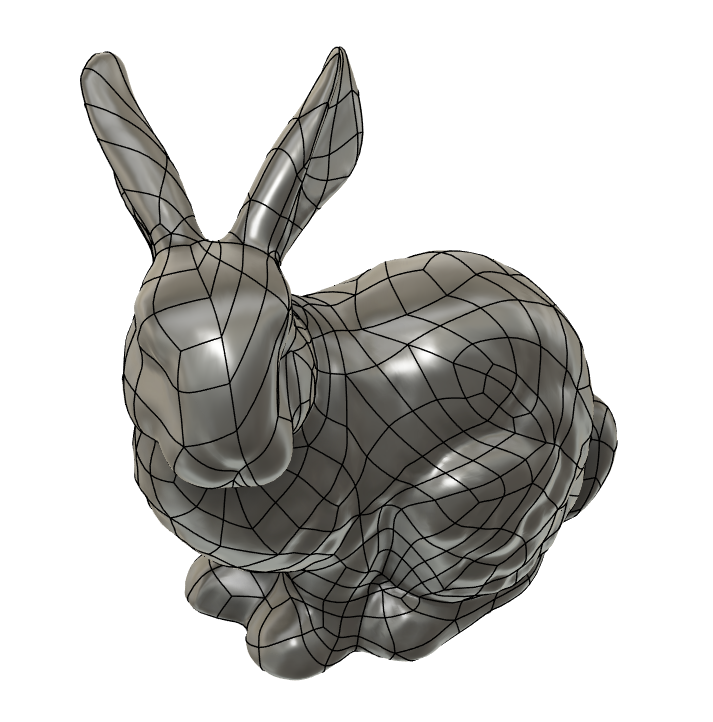}
        \end{minipage}
        \caption{}
    \end{subfigure}
    \hfill
    \begin{subfigure}[b]{0.22\textwidth}
        \begin{minipage}[t][\imgboxheight][t]{\textwidth}
            \centering
            \includegraphics[width=\textwidth]{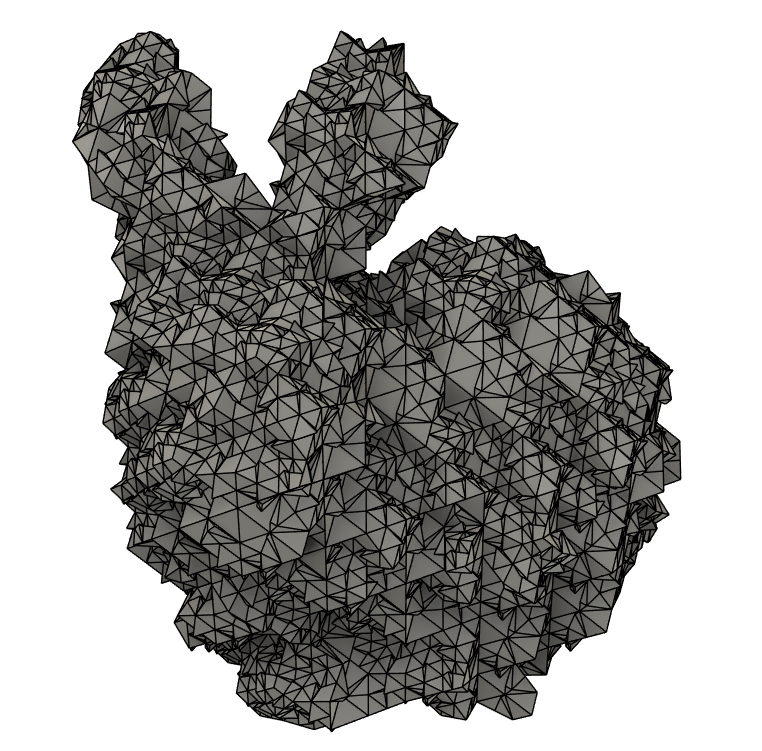}
        \end{minipage}
        \caption{}
    \end{subfigure}
    \hfill
    \begin{subfigure}[b]{0.22\textwidth}
        \begin{minipage}[t][\imgboxheight][t]{\textwidth}
            \centering
            \includegraphics[width=0.7\textwidth]{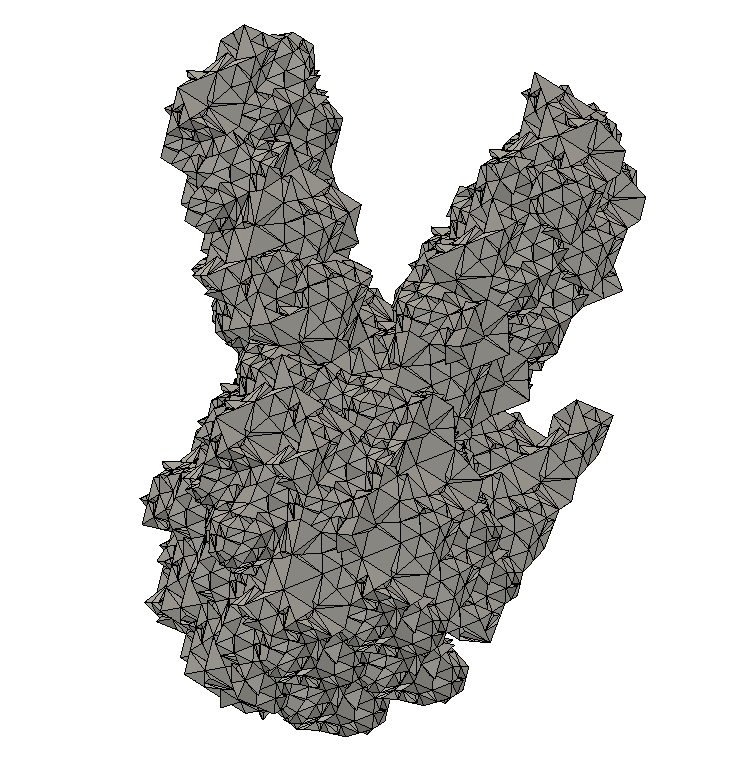}
        \end{minipage}
        \caption{}
    \end{subfigure}
    \hfill
    \begin{subfigure}[b]{0.22\textwidth}
        \begin{minipage}[t][\imgboxheight][t]{\textwidth}
            \centering
            \includegraphics[width=0.4\textwidth]{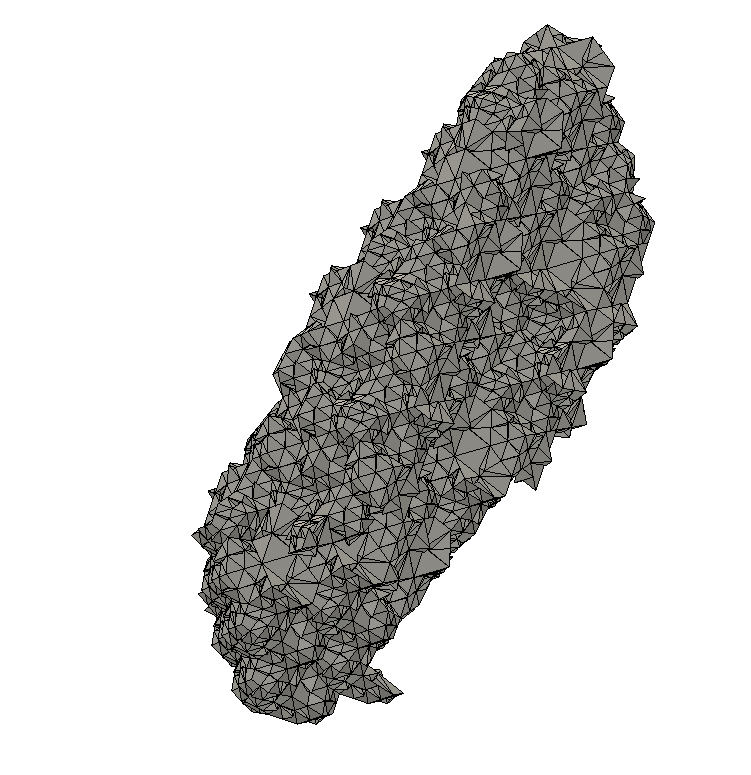}
        \end{minipage}
        \caption{}
    \end{subfigure}

    \caption{The original Stanford bunny (a) and a corresponding approximation given by $1323$ blocks (b) as well as an approximation of the head of the bunny made of $626$ blocks (c) and an ear made of $653$ blocks.}
    \label{fig:p23-bunny}
\end{figure}

Additionally the deformation can also be performed in a non-linear manner such that the resulting fundamental domain has curved faces. To illustrate this, we present a deformation of the same fundamental domain of group $p23$ in Figure~\ref{fig:smurf-195}. Note, that the resulting fundamental domain still retains the fundamental domain property and in particular can fill $\R^3$ without overlapping.
\begin{figure}[H]
    \centering
    \hfill
    \begin{subfigure}[b]{0.3\textwidth}
        \includegraphics[width=\textwidth]{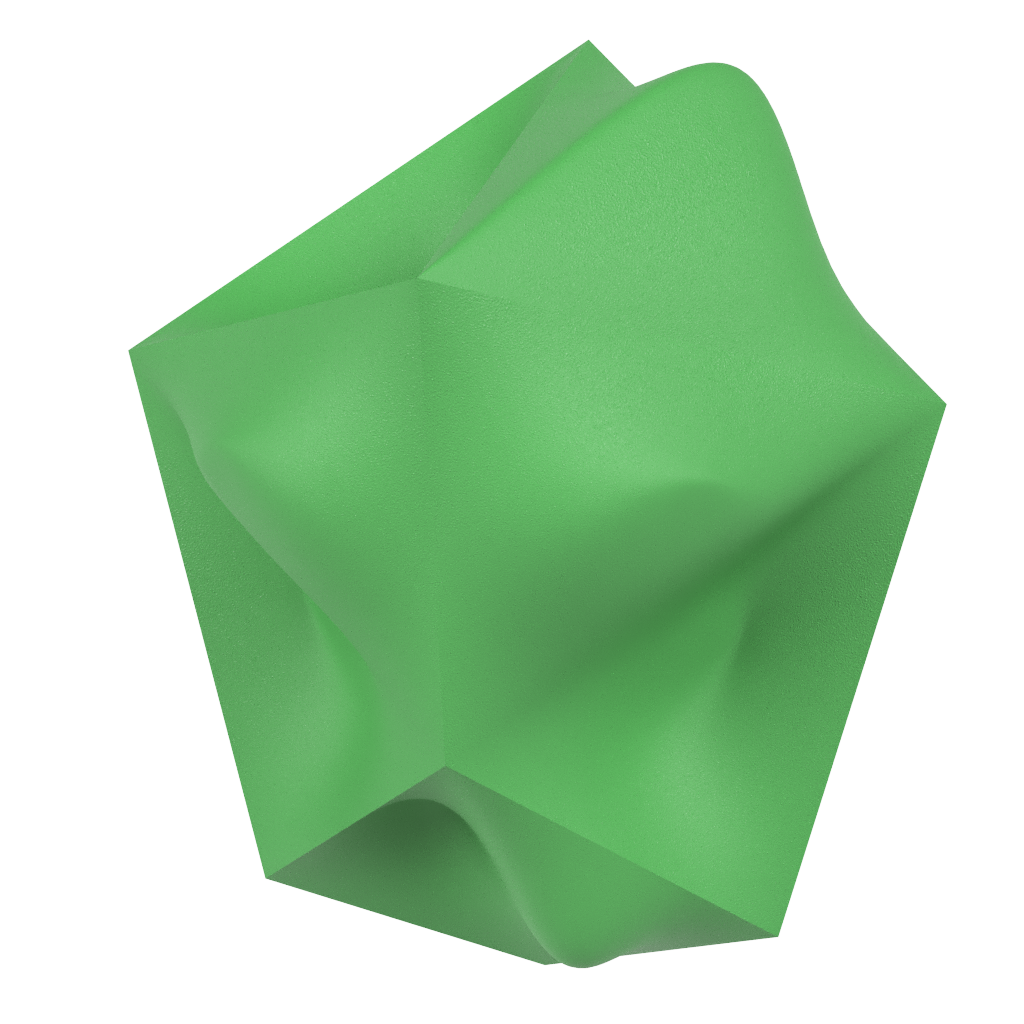}
    \end{subfigure}
    \hfill
    \begin{subfigure}[b]{0.3\textwidth}
        \includegraphics[width=\textwidth]{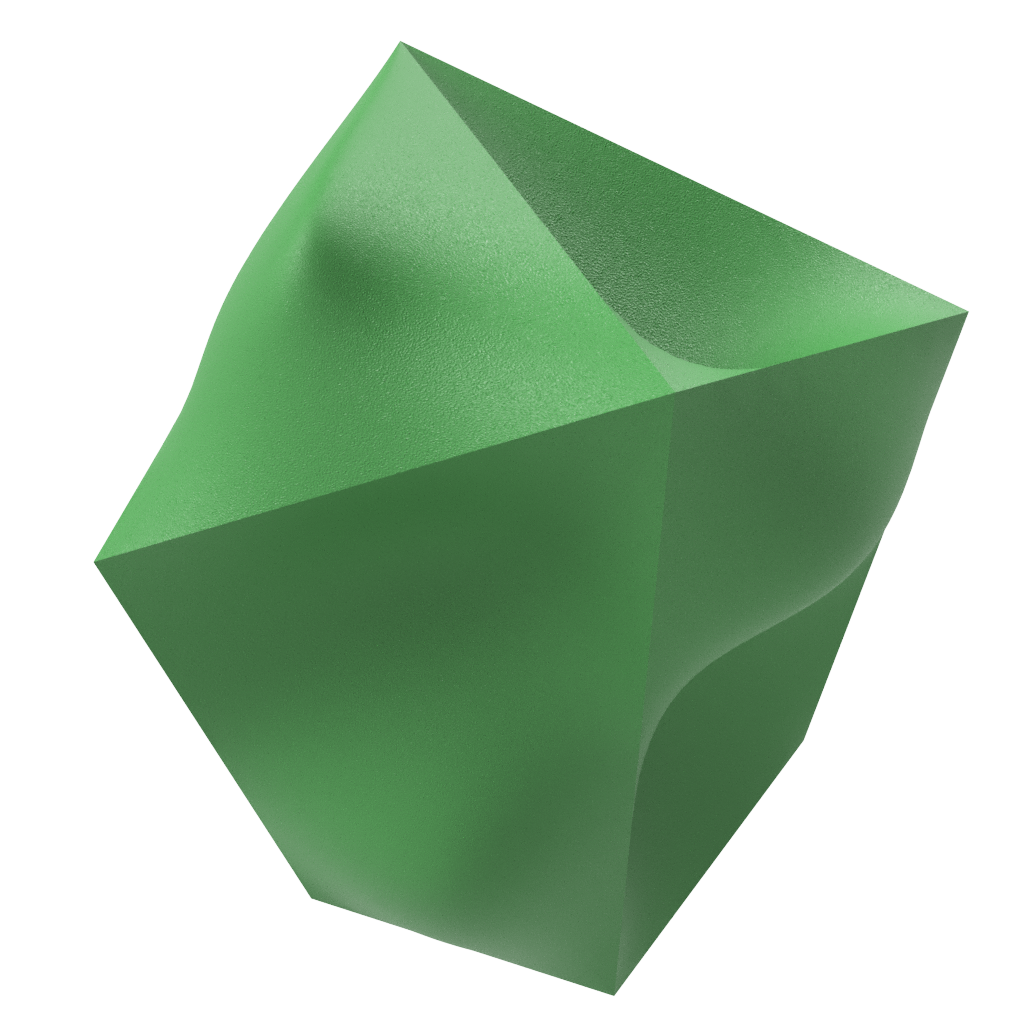}
    \end{subfigure}
    \hfill
    \begin{subfigure}[b]{0.3\textwidth}
        \includegraphics[width=\textwidth]{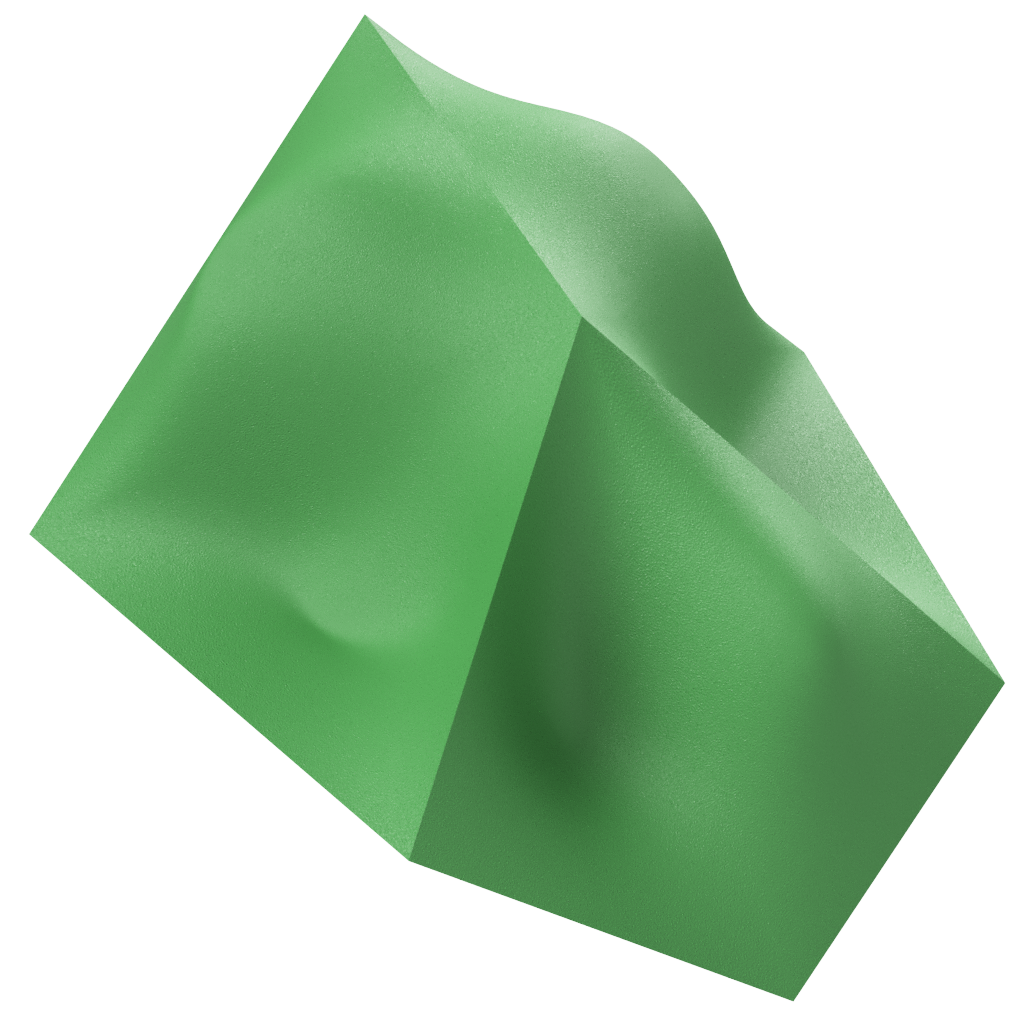}
    \end{subfigure}
    \hfill
    \caption{Different views of the curved deformed fundamental domain $F_{23}$ of $p23$.}
    \label{fig:smurf-195}
\end{figure}

 Using the crystallographic group $p23$ to arrange copies of $F_{23}$ in Euclidean $3$-dimensional space opens up the construction of various assemblies.
 To illustrate this instance, we present the assembly shown in Figure~\ref{fig:ass-195}. This assembly of blocks realises a tripod which we conjecture to be a TIA. Our hypothesis is strengthened by the fact that 3D-printed blocks with the frame (bottom blocks touching the ground) hold in place, suggesting they possess the interlocking property. 
Note, we have not proven that the two assemblies presented above are indeed TIAs. Here, we have illustrated these examples to showcase the rich design space of our approach. To verify the interlocking property of these assemblies, we have to employ expensive numerical computations as described in \cite{goertzen2025influence,Wang-2019-Topolock}. Establishing such a proof will be the focus of future work. 
\begin{figure}[H]
    \centering
    \hfill
    \begin{subfigure}[b]{0.335\textwidth}
        \includegraphics[width=\textwidth]{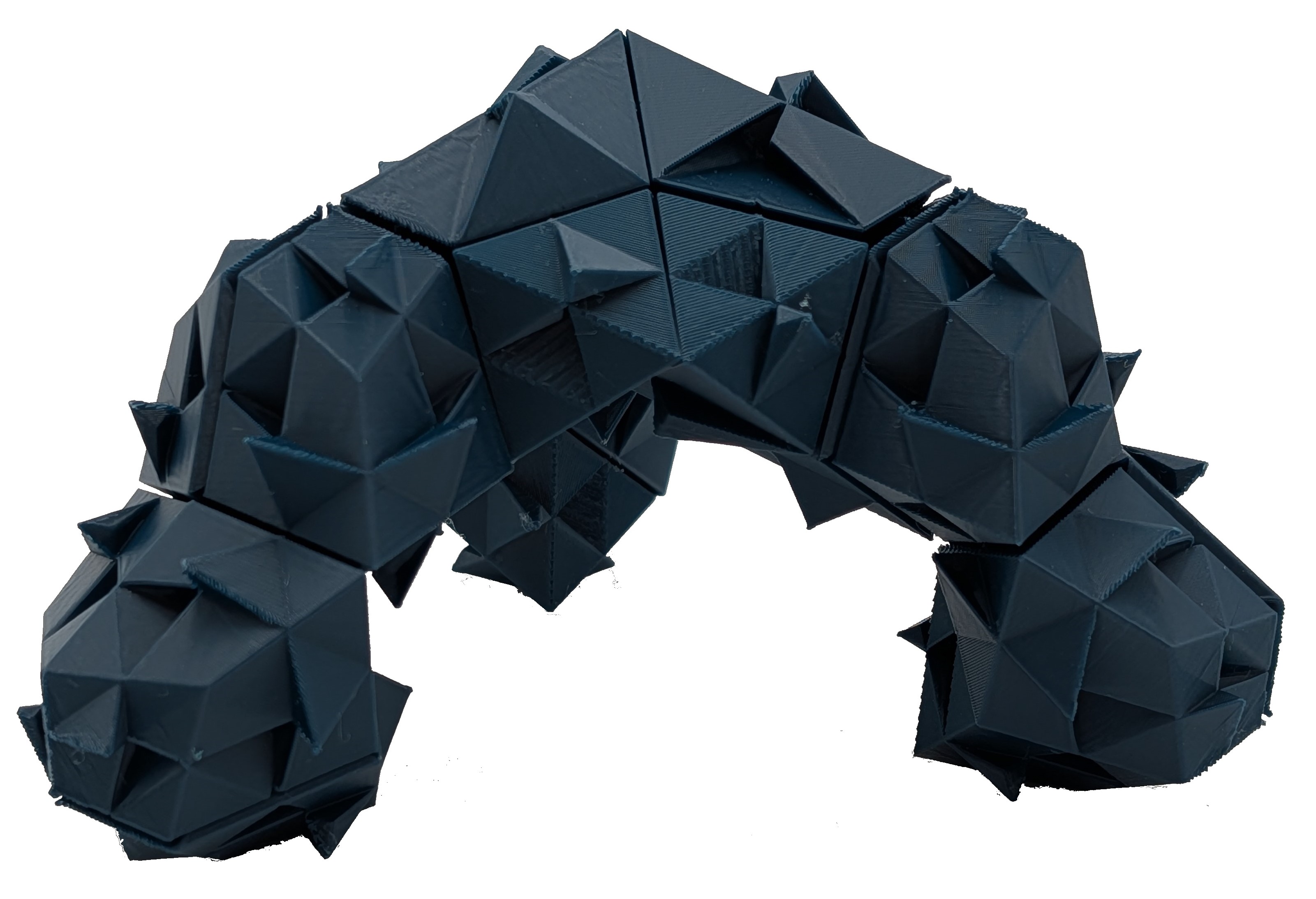}
    \end{subfigure}
    \hfill
    \begin{subfigure}[b]{0.3\textwidth}
        \includegraphics[width=\textwidth]{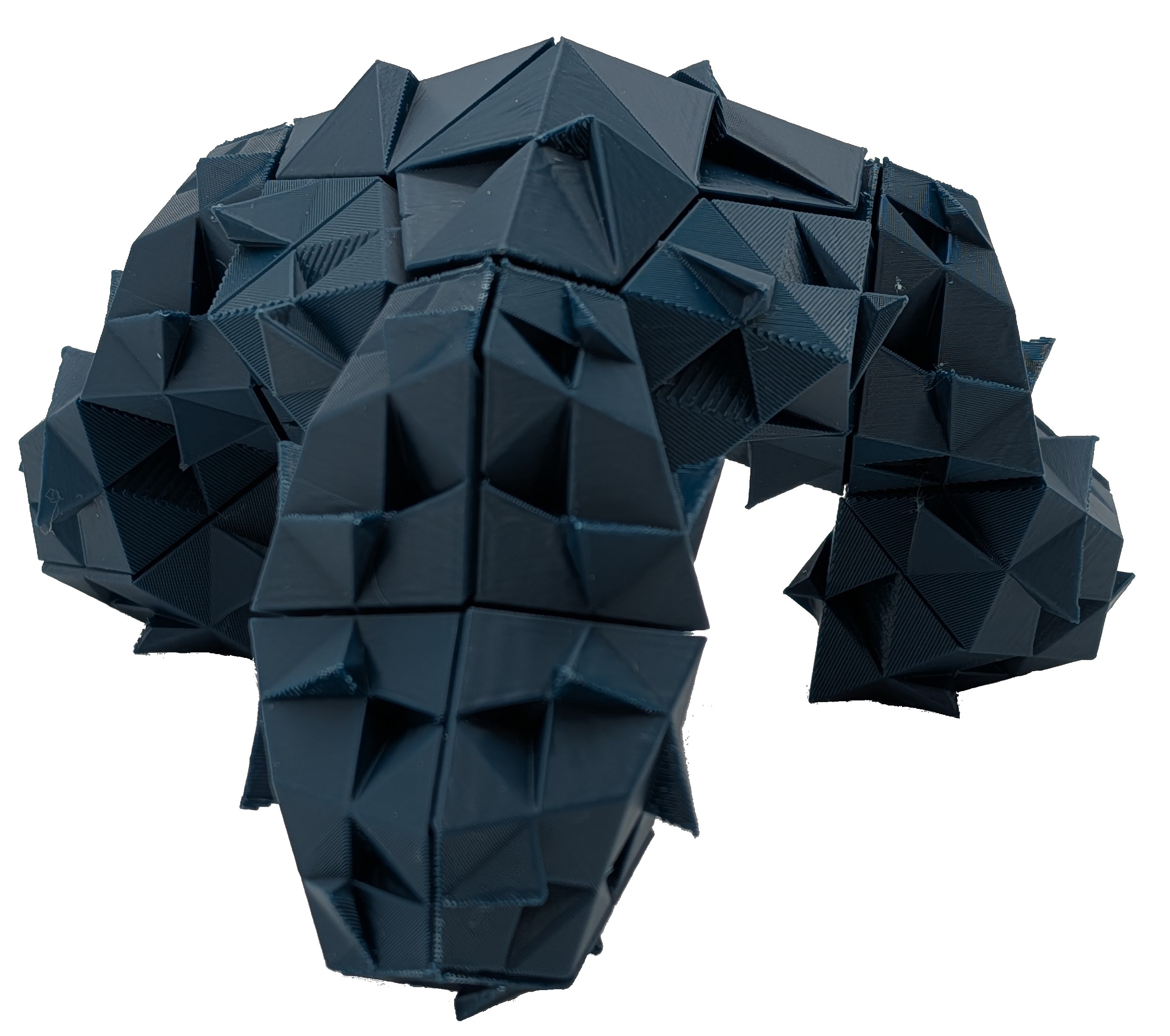}
    \end{subfigure}
    \hfill
    \begin{subfigure}[b]{0.27\textwidth}
        \includegraphics[width=\textwidth]{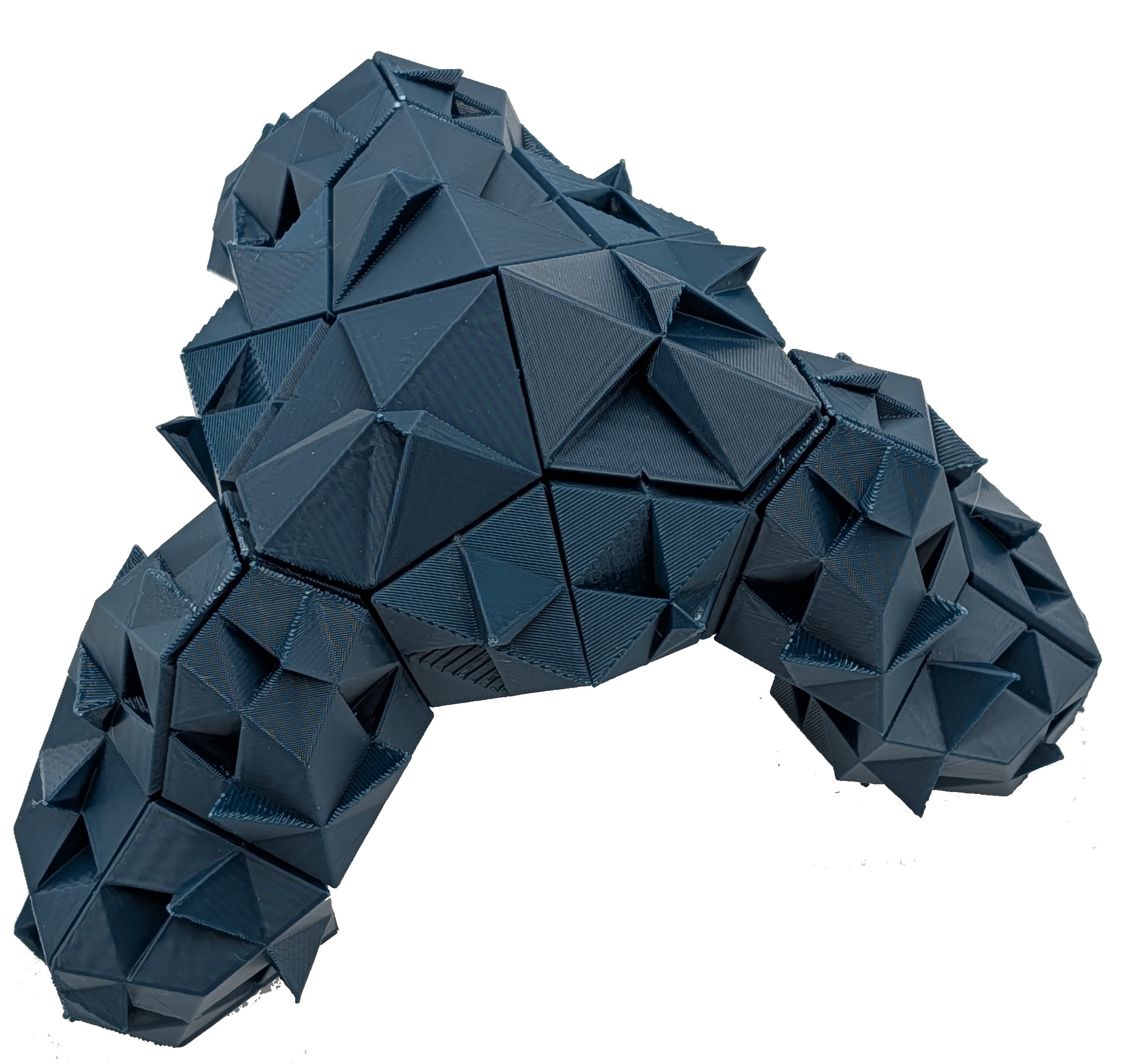}
    \end{subfigure}
    \hfill
    \caption{Different views of a 3D printed assembly of the deformation of the group $p23$. We conjecture that if the bottom blocks are fixed, this assembly is a TIA.}
    \label{fig:ass-195}
\end{figure}

\subsection*{Acknowledgement}
The authors would like to thank Bernd Souvignier for valuable discussions and for drawing our attention to some relevant literature.
The authors acknowledge the funding by the Deutsche Forschungsgemeinschaft (DFG, German Research Foundation) in the framework of the Collaborative Research Centre CRC/TRR 280 “Design Strategies for Material-Minimized Carbon Reinforced Concrete Structures – Principles of a New Approach to Construction” (project ID 417002380). Furthermore, R.\ Akpanya was
supported by a grant from the Simons Foundation (SFI-MPS-Infrastructure-00008650, JV).

\printbibliography[heading=bibintoc,title={Bibliography}]

\appendix

\section{Proofs for completeness}

\begin{theorem}{\cite[Thm. III.11 (ii)]{plesken1994}}\label{apx:thm:dirichlet-is-fund-dom}
	Let $\Gamma \leq \E(n)$ be a crystallographic group. If $u \in \R^n$ is a point in general position for $\Gamma$, then the Dirichlet cell $D(u, u^\Gamma)$ forms a fundamental domain for $\Gamma$.
\end{theorem}
\begin{proof}
	We first show that
	$$
		\R^n = \bigcup_{\gamma \in \Gamma} D(u, u^\Gamma)^\gamma.
	$$
By \Cref{lma:dirichlet-identity}, this is equivalent to proving  
	$$
		\R^n = \bigcup_{\gamma \in \Gamma} D(u^\gamma, u^\Gamma).
	$$
The union of Dirichlet cells on the right-hand side of the above equation forms a subset of $\R^n$. Hence, it remains to show that $\R^n \subseteq \bigcup_{\gamma \in \Gamma} D(u^\gamma, u^\Gamma) $ holds.
Let therefore $v \in \R^n$ be a point. We know that $v$ is inside the cell $D(u^{\gamma'},u^\Gamma)$, where $\gamma'\in \Gamma$ such that $d(u^{\gamma'}, v) \leq d(u^\beta, v)$ holds for all $\beta \in \Gamma$. This implies 
$$
    v \in D(u^{\gamma'}, u^\Gamma) \subseteq \bigcup_{\gamma \in \Gamma} D(u^\gamma, u^\Gamma) = \bigcup_{\gamma \in \Gamma} D(u, u^\Gamma)^\gamma.
$$ 

Next, we prove $D(u^\alpha, u^\Gamma) \cap D(u^\beta, u^\Gamma) =\emptyset$ for distinct $\alpha, \beta \in \Gamma$.
We show this statement by contradiction and assume that there exist distinct $\alpha, \beta \in \Gamma$ such that there exists $v \in \left( D(u^\alpha, u^\Gamma) \right)^\circ \cap \left(D(u^\beta, u^\Gamma)\right)^\circ$. This implies $d(v, u^\alpha) < d(v, u^\gamma)$ for all $\gamma \in \Gamma \setminus \{ \alpha \}$ and $d(v, u^\beta) < d(v, u^\delta)$ for all $\delta \in \Gamma \setminus \{ \beta \}$. Since these inequalities hold for all elements in $\Gamma$, the inequalities $d(v, u^\alpha) < d(v, u^\beta)$ and $d(v, u^\beta) < d(v, u^\alpha)$ follow. This is a contradiction.\\
Now, we have to verify that $D(u,u^\Gamma)$ contains a system of representatives of the action of $\Gamma$ on $\R^n$. We start by showing that at least one point of each orbit is contained in $D(u,u^\Gamma).$ For this, let $v \in \R^n$ be a point and $w \in v^\Gamma$ the point satisfying
	$$
		w \coloneqq \argmin_{\hat{w} \in v^\Gamma} \{ d( \hat{w}, \hat{u} ) \mid \hat{u} \in u^\Gamma \}.
	$$
    That means, $w$ is the point in the orbit of $v$
	 with minimum distance to the orbit of $u$.
    Furthermore, let $\alpha \in \Gamma$ be the group element so that $w$ is closest to $u^\alpha$, i.e.\ $d(w, u^\alpha) \leq d(w, u^\beta)$ for all $\beta \in \Gamma$. Hence, $w$ lies in the Dirichlet cell $D(u^\alpha, u^\Gamma)$. Since $\alpha$ is an isometry, $d(w^{\alpha^{-1}}, u) \leq d(w^{\alpha^{-1}}, (u^\beta)^{\alpha^{-1}})$ holds for all $\beta \in \Gamma$. The above inequality holds for all $\beta \in \Gamma$ and  $w^{\alpha^{-1}}$ is therefore contained in the Dirichlet cell around $u$. This means  $w^{\alpha^{-1}} \in D(u, u^\Gamma)$.
    Since $w^{\alpha^{-1}}\in v^\Gamma$, we conclude that $D(u, u^\Gamma)$ contains a representative $V$ of each orbit of $\Gamma$ acting on $\R^n$.
    It remains to show that $V$ also contains the interior of $D(u, u^\Gamma)$.
    This means that no two points in the interior of $D(u, u^\Gamma)$ lie in the same orbit of the action of $\Gamma$ on $\R^n$.
    Again, we prove this by contradiction and assume that there exist distinct $v, w \in \left( D(u, u^\Gamma) \right)^\circ$ and $\gamma \in \Gamma$ such that $v^\gamma = w$. The interiors of the Dirichlet cells are pairwise disjoint sets, i.e.\ $\left(D(u, u^\Gamma)\right)^\circ$ and $\left(D(u^\gamma, u^\Gamma)\right)^\circ$ are either disjoint or equal. Since $\left(D(u^\gamma, u^\Gamma)\right)^\circ = \left( \left(D(u, u^\Gamma)\right)^\circ \right)^\gamma$ and $v \in \left(D(u^\gamma, u^\Gamma)\right)^\circ$, $w = v^{\gamma^{-1}} \in \left( \left(D(u, u^\Gamma)\right)^\circ \right)^{\gamma^{-1}}$, the intersection $D(u,u^\Gamma)^\circ\cap D(u^\gamma,u^\Gamma)^\circ $ is not empty and the Dirichlet cells have to be equal. This implies that $u = u^\gamma$. Since $u$ is in general position for $\Gamma$, this implies $\gamma = \Id.$ This is a contradiction to $v \neq w$ and we conclude the proof.
\end{proof}

\begin{proposition}\label{apx:prop:translation-cell-fund-dom}
    Let $\Gamma\leq \E(n)$ be a crystallographic group and $S=\{\rho_i\mid i \in I\}\cup \{\tau_k\mid k \in K\}$ a generating set of $\Gamma$ as described in \Cref{not:cryst-grp-cosets}. If $F$ is a fundamental domain for $\Gamma,$ then
    $$
		C \coloneqq \bigcup_{i \in I} F^{\rho_i}
	$$
	forms a fundamental domain for $\T(\Gamma) \leq \E(n)$. 
\end{proposition}
\begin{proof}
First, we recall from \Cref{not:cryst-grp-cosets} that $\Gamma = \bigcup_{i \in I} \rho_i\T(\Gamma)$ holds. This implies 
$$
    \R^n = \bigcup_{i \in I}\bigcup_{\tau \in \mathcal{T}(\Gamma)}F^{\rho_i\tau}=\bigcup_{\tau \in \mathcal{T}(\Gamma)} {\left(\bigcup_{i \in I}F^{\rho_i}\right)}^{\tau} =\bigcup_{\tau \in \mathcal{T}(\Gamma)} C^\tau. 
$$ 
Next, we establish the existence of a system of representatives $V'$ for the orbits of $\T(\Gamma)$ acting on $\R^n$ satisfying $C^\circ\subseteq V'\subseteq C$. Let therefore $V$ denote the system of representatives for the orbits of $\Gamma$ acting on $\R^n$ with $F^\circ \subseteq V \subseteq F$. The existence of $V$ implies $\bigcup_{i \in I} (F^{\circ})^{\rho_i} \subseteq \bigcup_{i \in I} V^{\rho_i} \subseteq \bigcup_{i \in I} F^{\rho_i} = C$.
This allows us to define $V'$ as
$$
    V' \coloneqq \bigcup_{i \in I} V^{\rho_i}.
$$ 
We claim that $V'$ is a system of representatives for the orbits of $\T(\Gamma)$ acting on $\R^n$. First, we have to show that $V'$ contains at least one point of each $\mathcal{T}(\Gamma)$-orbit.
So, let $x \in \R^n$ be an arbitrary point. Since $V$ is a system of representatives of the orbits of $\Gamma$ acting on $\R^n$, there exist $z \in V$ and $\gamma \in \Gamma$ such that $x = z^\gamma$. Thus, there are $i \in I$ and $\tau \in \T(\Gamma)$ such that $\gamma = \rho_i \tau$.
If we define $y \coloneqq z^{\rho_i}$, it follows that $y \in x^{\T(\Gamma)}$. Because of $y \in V^{\rho_i}$, the statement $y \in V'$ follows.\\
Next, we show that $V'$ contains at most one point of each orbit of $\T(\Gamma)$ acting on $\R^n$. 

We assume that there are $x, y \in F^\circ$ with $x \neq y$ and $\tau \in \T(\Gamma)$ such that $x^\tau = y$. 

By definition of $V'$, we know that there exists $\rho_i$ and $\rho_j$ such that $x\in V^{\rho_i}$ and $y\in V^{\rho_j}.$
Again, since $V$ is a system of representatives of the action of $\Gamma$ on $\R^n$, there exist $a, b \in F^\circ$ such that $a^{\rho_i} = x$ and $b^{\rho_j} = y$. This implies that $a^{\rho_i \tau \rho_j^{-1}} = x^{\tau \rho_j^{-1}} = y^{\rho_j^{-1}} = b$. Since $a$ and $b$ are contained in the fundamental domain $V$, we get that $a = b$ and thus $\rho_i \tau \rho_j^{-1} = Id$. As $\rho_i \tau = \rho_j$ and $\{ \rho_i \}_{i \in I}$ are representatives of the cosets this implies that $i = j$ and therefore $x = a^{\rho_i} = b^{\rho_j} = y$ which is a contradiction.

\end{proof}

\end{document}